\documentclass[final]{siamltex}

\usepackage{amsmath,amssymb,amsfonts}
\usepackage{enumerate}
\usepackage{picinpar}
\usepackage{layout}
\usepackage{verbatim}
\usepackage{epsfig}
\usepackage{graphicx}
\usepackage{pgf}
\newtheorem{remark}{Remark}
\newtheorem{example}[theorem]{Example}
\newcounter{fig}

\def \eqdef {\overset{\text{\tiny def}}{=}}

\def\R{\mathbb R}
\def\Z{\mathbb Z}

\def\C{\mathcal C}
\def\S{\mathcal S}
\def\Re{\mathcal R}
\def\K{\mathcal K}

\newcommand{\ka}{\kappa}

\title{A proof of the Global Attractor Conjecture in the single linkage class case}
\author{David F. Anderson$^{1}$}

\begin{document}

\maketitle

\footnotetext[1]{Department of Mathematics, University of Wisconsin, Madison,
  WI, 53706.  Grant support from NSF grant DMS-1009275}

\begin{abstract} 
  This paper is concerned with the dynamical properties of deterministically modeled chemical reaction systems.  Specifically, this paper provides a proof of the Global Attractor Conjecture in the setting where the underlying reaction diagram consists of a single linkage class, or connected component.  The conjecture dates back to the early 1970s and is the most well known and important open problem in the field of chemical reaction network theory.  The resolution of the conjecture has important biological and mathematical implications in both the deterministic and stochastic settings.  One of our main analytical tools, which is introduced here, will be a method for partitioning the relevant monomials of the dynamical system along sequences of trajectory points into classes with comparable growths.  We will use this method to conclude that if a trajectory converges to the boundary, then a whole \textit{family} of Lyapunov functions decrease along the trajectory. This will allow us  to overcome the fact that the usual Lyapunov functions of chemical reaction network theory are bounded on the boundary of the positive orthant, which has been the technical sticking point to a proof of the Global Attractor Conjecture in the past.  
\end{abstract}

  \begin{keywords} persistence, global stability, population processes, chemical reaction systems, 
  mass-action kinetics, deficiency, complex-balancing, 
  detailed-balancing
\end{keywords}

\begin{AMS}
37C10, 80A30, 92C40, 92D25
\end{AMS}

\section{Introduction}
This paper is concerned with the qualitative behavior of deterministically modeled chemical reaction systems with mass-action kinetics.  We will provide multiple results pertaining to weakly reversible reaction systems that will allow us to conclude that the Global Attractor Conjecture, the most well known and important open problem in the field of chemical reaction network theory, holds in the single linkage class case.  That is, the conjecture holds when the underlying reaction diagram consists of a single connected component.

\subsection{Background and statement of the problem}
\label{sec:background}

 Natural questions about the
qualitative behavior of deterministically modeled chemical reaction systems include the existence of positive
equilibria (fixed points), stability properties of equilibria, and the non-extinction, or persistence, of species, which are the constituents of the system.  As the exact values of key system parameters, termed \textit{rate constants} and which we will denote by $\kappa_k$, are usually difficult to find experimentally
and, hence, are oftentimes unknown, it would be best to answer these
questions \textit{independently of the values of these parameters}.
Building off the work of Fritz Horn, Roy Jackson, and Martin Feinberg
\cite{FeinbergLec79, Feinberg87, FeinHorn1972, Horn72, Horn74, HornJack72} the
mathematical theory termed ``Chemical Reaction Network Theory'' has
been developed over the previous thirty-five years to answer these
types of questions.

Early work by Feinberg, Horn, and Jackson showed that if a reaction network with
deterministic mass-action kinetics admits a so called ``complex-balanced''
equilibrium, see Definition \ref{def:cb},
then there exists a unique complex-balanced equilibrium within the
interior of \textit{each} positive compatibility class, or invariant manifold \cite{FeinHorn1972, Horn72,
  HornJack72} (throughout, we will often refer to the invariant manifolds of the systems of interest as \textit{compatibility classes}, or \textit{stoichiometric} compatibility classes, to stay in line with the terminology of chemical reaction network theory).  Horn and Jackson also
proved the existence of a strict, entropy type, Lyapunov function
that gives local asymptotic stability of each such equilibrium
\textit{relative to its compatibility class}.  Later, Horn, Jackson, and Feinberg proved what is best known as the Deficiency Zero Theorem: 
that \textit{regardless of the choice of parameters $\kappa_k$}, a
reaction network with deterministic mass-action kinetics that is both
weakly reversible and has a deficiency of zero must admit a
complex-balanced equilibrium \cite{FeinbergLec79, Feinberg87, FeinHorn1972}.  Here, a reaction network is \textit{weakly reversible} if each of the connected components of its reaction diagram is strongly connected, see Definition \ref{def:WR}, and the deficiency of a network is defined in Definition \ref{def:deficiency}.  
Collecting ideas shows that the results pertaining to complex-balanced systems, for example those in  \cite{FeinHorn1972, Horn72,
  HornJack72}, apply to this (deficiency zero and weakly reversible) large class
of networks.

It was conjectured at least as early as 1974 that complex-balanced
equilibria of reaction networks are globally asymptotically stable
relative to the interior of their positive compatibility classes
\cite{Horn74}.  This problem was given the name ``Global Attractor
Conjecture'' by Craciun \textit{et al.}\ \cite{CraciunShiu09}, and is considered
to be one of the most important open problems in the field of chemical
reaction network theory \cite{Anderson08, AndShiu,
  CraciunShiu09, CraciunPantea, Sontag2001}.  \vspace{.1in}

\noindent \textbf{Global Attractor Conjecture.}  A complex-balanced
equilibrium contained in the interior of a positive compatibility
class  is a \textit{global attractor} of the interior of that positive class.
\vspace{.1in}

Using the Lyapunov function of Horn and Jackson it is relatively straightforward to show that each trajectory of a complex-balanced
system remains bounded and converges either to the unique equilibrium within the interior of its invariant manifold, or to the boundary of the positive orthant, $\partial \R^N_{\ge 0}$.  Therefore, the Global Attractor Conjecture will be proven if it can be shown that any complex-balanced system is
\textit{persistent} in the sense of Definition \ref{def:persistence} below.

\begin{definition} 
  For $t\ge 0$ denoting time, let $\phi(t,x_0)$ be a trajectory to a dynamical system in $\R^N$ with initial condition $x_0$.  A trajectory $\phi(t,x_0)$ with state space $\R^N_{\ge 0}$  is said to be {\em 
    persistent} if 
    \begin{equation*}
	\liminf_{t\to \infty} \phi_i(t,x_0) > 0,
\end{equation*}
for all $i \in \{1,\dots,N\}$, where $\phi_i(t,x_0)$ denotes the $i$th component of $\phi(t,x_0)$. A dynamical system is  said to be
   {\em persistent} if each trajectory with positive initial 
  condition is persistent.  
  \label{def:persistence} 
\end{definition} 

We will use the notation $\phi(t,x_0)$ for trajectories  throughout the paper.
 We see in Definition \ref{def:persistence} that persistence corresponds to a non-extinction requirement.  Some authors refer to dynamical systems satisfying  the above condition  as \textit{strongly persistent} \cite{Takeuchi1996}.  In their work, persistence only requires the weaker condition that $\limsup_{t\to \infty} \phi_i(t,x_0) > 0$  for each $i\in\{1,\dots,N\}$.

\begin{definition}
For $t\ge 0$, let $\phi(t,x_0)$ be a trajectory to a dynamical system in $\R^N$ with initial condition $x_0$.  The set of 
{\em $\omega$-limit points} for this trajectory is the set of 
accumulation points: 
\begin{equation*} 
  \omega(\phi(\cdot ,x_0))  \eqdef  \{x \in \R^N ~:~ \phi(t_n,x_0) \to {x}  \text{ 
    for some sequence } t_n \to \infty \}. 
\end{equation*} 
%For autonomous dynamical systems we will simply write $\omega(x_0)$.
\end{definition}
\indent Note that for bounded trajectories, persistence is equivalent to the condition that  $\omega(\phi(\cdot,x_{0})) \cap \partial \R^{N}_{\ge 0} = 
  \emptyset$.

It can
be shown that a complex-balanced network is necessarily weakly
reversible \cite{Feinberg72, FeinbergLec79,Gun2003}.   Therefore, in light of the above
discussion, the Global Attractor Conjecture is implied by the
following, more general, conjecture of Feinberg (see Remark 6.1.E in \cite{Feinberg87}):

\vspace{.1in}

\noindent \textbf{Persistence Conjecture.}  Any weakly reversible
reaction network with mass-action kinetics and bounded trajectories is persistent.

\vspace{.1in}

\noindent Other formulations of the Persistence Conjecture leave out the assumption of bounded trajectories, and the above is, therefore, a weaker version of the usual conjecture.  In fact, it is an open problem as to whether or not weakly reversible networks give rise to only bounded trajectories, and we feel it is best to separate these two conjectures.  Note that the Persistence Conjecture makes no assumption on the choice
of rate constants.

Both conjectures remain open. However, in recent years there has been much activity aimed at their resolution.   It is known that only certain faces of the boundaries of the invariant manifolds can admit $\omega$-limit points, those associated with a \textit{semi-locking set} (using the terminology of \cite{Anderson08}), which is a subset of the species whose absence is forward invariant. (Semi-locking sets were termed \textit{siphons} in the earlier paper \cite{Angeli2007} in which their concept was formally introduced.  However, see Proposition 5.3.1 and Remark 6.1.E of \cite{Feinberg87} for an earlier treatment that anticipated these definitions.)  This fact has typically focused attention on understanding the behavior of these systems near different faces of the boundaries of the invariant manifolds.  For example, Anderson \cite{Anderson08}  and Craciun, Dickenstein, Shiu, and Sturmfels \cite{CraciunShiu09} used different methods to independently conclude that  vertices of the positive compatibility classes (which are polyhedra, see \cite{AndShiu}) can not be $\omega$-limit points even if they are associated with a semi-locking set.  In \cite{AndShiu}, it was shown that weak reversibility of the network
guarantees that facets---faces of one dimension less than the
compatibility class itself; that is, a face of codimension one---of
the positive classes  ``repel,'' in a certain sense, trajectories.  This fact was used to prove the Global Attractor Conjecture when the stoichiometric compatibility classes, or invariant manifolds, are  two-dimensional. More recently, Craciun, Pantea, and Nazarov proved that two-species, weakly reversible systems are both persistent and permanent (trajectories eventually enter a fixed, compact subset of the strictly positive orthant $\R^N_{> 0}$). They then used this fact to prove that the Global Attractor Conjecture holds for three-species systems
\cite{CraciunPantea}.  Pantea then extended these ideas to prove the Global Attractor Conjecture for all systems for which the stoichiometric compatibility class has dimension less than or equal to three \cite{Pantea2011}.  

In \cite{Sontag2007} the authors studied persistence by introducing the notion of \textit{dynamic non-emptiability} for semi-locking sets, which corresponds to a dominance ordering of the monomials near a given face of the compatibility class.  These ideas were expanded in \cite{JohnstonSiegel2011} where the concept of \textit{weak dynamic non-emptiability} was introduced, and a connection was made to the work on facets in \cite{AndShiu}.  This work should also be compared to the use of \textit{strata} in both \cite{JohnstonSiegel2} and \cite{CraciunShiu09}, where monomial dominance is again considered near faces of the invariant manifold.  Later, we will see that monomial dominance is at the heart of the current paper as well.  However, and importantly, the dominance is no longer sequestered to individual faces of the invariant manifold and is instead considered along sequences of trajectory points in time.  Further, the monomials are grouped into classes of comparable growth, which allows for a greater understanding of the behavior of the system.

 Biological models
in which the non-existence of $\omega$-limit points on the boundary implies global
convergence  include the ligand-receptor-antagonist-trap model of
G. Gnacadja \textit{et al.}  \cite{Gnacadja2007, Gnacadja2009}, the
enzymatic mechanism of D. Siegel and D. MacLean \cite{Siegel04}, and
T. McKeithan's T-cell signal transduction model \cite{McKeithan95}
(the mathematical analysis appears in the work of E. Sontag
\cite{Sontag2001} and Section 7.1 in the Ph.D. thesis of M.  Chavez
\cite{ChavezThesis}).  Recently, Gopalkrishnan showed that any network that violates the Persistence Conjecture must be `catalytic' in a precise sense \cite{Manoj2011}.

Complex-balanced systems also play an important role in the study of stochastically modeled reaction networks.  In \cite{AndProdForm}, Anderson, Craciun, and Kurtz showed that a stochastically modeled system admits a product form stationary distribution if the associated deterministically modeled system admits a complex-balanced equilibrium (see also \cite{Lubensky2010} in which the same results were derived independently by David Lubensky).  Also, in the study of stochastically modeled reaction systems with multiple scales, which is the norm as opposed to the exception in the stochastic setting, it is often desirable to perform a reduction to the system using asymptotic analysis (usually averaging and law of large number techniques), see, for example, \cite{Ball06} or \cite{KurtzKang}.  If one component of the system can be shown to behave deterministically in the asymptotic limit, then knowing this component converges to a steady state (i.e. knowing that the conclusions of the current paper hold) may allow for the proof of results pertaining to the dynamics of the other components.  

\subsection{Results in this paper}

 In this paper, we will provide multiple results pertaining to deterministically modeled, weakly reversible systems.  Most  results will pertain to the dynamics of individual trajectories.  These results will allow us to conclude that the Global Attractor Conjecture holds in the case when the underlying reaction network consists of one linkage class, or connected component.  It is worth noting that we will not provide a proof of the Persistence Conjecture in the one linkage class case.  As will become apparent, the technical difference between the conjectures will be captured by a condition on where the $\omega$-limit points of a trajectory can reside, see Theorem \ref{thm:main}.
 
 To prove our results, we will introduce a method for partitioning the relevant monomials of the dynamical system along sequences of trajectory points into classes with comparable growths.  This method will allow us to conclude that if a trajectory converges to the boundary, then a whole \textit{family} of Lyapunov functions decrease along the trajectory.  We will then be able to overcome the fact that the usual Lyapunov functions of chemical reaction network theory are bounded on the boundary of the positive orthant, which has been the technical sticking point to a proof of the Global Attractor Conjecture in the past.  The methods developed should prove useful in future contexts, both deterministic and stochastic, as well as the current one.  Also, it will be natural to focus our attention on systems with a generalized mass-action kinetics in which the rate constants are allowed to be functions of time.  This context is useful because the projection of a trajectory of a reaction system onto some relevant subset of the species can itself be viewed as a trajectory of a reaction system with generalized mass-action kinetics.

The outline of the paper is as follows.  In Section \ref{sec:def_concepts}, we will provide the requisite definitions and terminology from chemical reaction network theory.  In Section \ref{sec:projection}, we will discuss projected dynamics, and introduce and develop the basic properties of reduced reaction networks and generalized mass-action systems.  In Section \ref{sec:results}, we will  give our main results together with their proofs.  In Section \ref{sec:example}, we present an example to demonstrate our results.

\section{Preliminary concepts and definitions}
\label{sec:def_concepts}

Most of the following definitions are standard in chemical reaction network theory.  The interested reader should see  \cite{FeinbergLec79} or \cite{Gun2003} for a more detailed introduction.\vspace{.125in}

\noindent \textbf{Reaction networks.}  An example of a chemical reaction is $2S_{1}+S_{2} ~\rightarrow~ S_{3},$
where we interpret the above as saying two molecules of type $S_1$ combine with a molecule of type $S_2$ to produce a molecule of type $S_3$.  For now, assume that there are no other reactions under consideration.  The $S_{i}$ are called chemical {\em species} and the linear combinations of the species found at either end of the reaction arrow, namely $2S_{1}+S_{2}$ and
$S_{3}$, are called chemical {\em complexes.}  Assigning the {\em
  source} (or reactant) complex $2S_{1}+S_{2}$ to the vector $y =
(2,1,0)$ and the {\em product} complex $S_{3}$ to the vector
$y'=(0,0,1)$, we can formally write the reaction as $ y \rightarrow y' .$ 

 In the
general setting we denote the number of species by $N$, and for $i \in \{1,\dots, N\}$ we denote the $i$th species as $S_{i}$.  We then
consider a finite set of reactions with the $k$th denoted by $ y_{k} \rightarrow y_{k}', $
 where $y_k, y_k' \in \Z^N_{\ge 0}$ are (non-equal) vectors whose components give the coefficients of the source and product complexes, respectively.  Using a slight abuse of notation, we will also refer to the vectors $y_k$ and $y_k'$ as the complexes.  Note that if $y_k = \vec 0$ or $y_k' = \vec 0$ for some $k$, then the $k$th
reaction represents an input or output, respectively, to the system.  Note also that any
complex may appear as both a source complex and a product complex in
the system.  We will usually, though not always (for example, see condition 3 in Definition \ref{def:crn} below) use the prime $'$ to denote the product complex of a given reaction.

As an example, suppose that the entire system consists of the two species $S_1$ and $S_2$ and the two reactions
\begin{equation}
	S_1 \to S_2 \quad \text{and} \quad S_2 \to S_1,
	\label{eq:ex1}
\end{equation}
where $S_1 \to S_2$ is arbitrarily labeled as ``reaction 1.''  Then $N = 2$ and 
\begin{equation*}
	y_1 = (1,0), \quad y_1' = (0,1) \qquad \text{and} \qquad  y_2 = (0,1), \quad y_2' = (1,0).
\end{equation*}
Thus, the vector $(1,0)$, or equivalently the complex $S_1$, is both $y_1$, the source of the first reaction, and $y_2'$, the product of the second.

 For ease of notation, when there is no need for
enumeration we will typically drop the subscript $k$  from the notation
for the complexes and reactions.

\begin{definition}
  Let $\S = \{S_i\}_{i=1}^N$, $\C = \{y\}$ with $y \in \Z^N_{\ge 0}$, and $\Re = \{y \to y'\}$ denote
  finite sets of species, complexes, and reactions, respectively.  The triple
  $\{\S, \C, \Re\}$ is called a {\em chemical reaction network} so long as the following three natural requirements are met:
  \begin{enumerate}
     \item  For each $S_i\in \S$, there exists at least one complex $y\in \C$  for which $y_{i} \ge 1$.
     \item There is no trivial reaction $y \to y \in \Re$ for some complex $y \in \C$.
     \item For any $y\in \C$, there must exist a $y'\in \C$ for which $y \to y' \in \Re$ or $y' \to y \in \Re$.
  \end{enumerate}
    %Throughout, we will use $N$ and $R$ to denote the number of elements of $\S$ and $\Re$, respectively.  
    If the triple $\{\S,\C,\Re\}$ satisfies all of the above requirements except 1., above, then we say $\{\S,\C,\Re\}$ is a  chemical reaction network {\em with inactive species}.
  \label{def:crn}
\end{definition}

\textbf{Notation:}  We will use each of the following choices of notation to denote a complex from $\C$:  $y$, $y'$, $y_k$, $y_k'$, $y_i$, $y_j$, $y_{\ell}$, and even $z_k$.  However, there will be other times in which we wish to denote the $i$th component of a complex.  If the complex in question has been denoted by $y_k$, then we will write $y_{k,i}$.  However, if the complex has been denoted by $y$, then we would write its $i$th component as $y_i$, which, through context, should not cause confusion with a choice of \textit{complex} $y_i$.  See, for example, condition 1 in Definition \ref{def:crn} above.

\begin{definition}
To each reaction network $\{\mathcal{S},\mathcal{C},\mathcal{R}\}$
we assign a unique directed graph called a {\em reaction diagram}
constructed in the following manner.  The nodes of the graph are the
complexes, $\mathcal{C}$.  A directed edge $(y,y')$ exists if and only
if $y \to y' \in \Re$.  Each connected
component of the resulting graph is termed a  {\em linkage class} of
the graph.  %, and we denote the number of linkage classes by~$l$
\label{def:diagram}
\end{definition}

For example, the system described in and around \eqref{eq:ex1} has reaction diagram $S_1 \rightleftarrows S_2$,
which consists of a single linkage class.

\begin{definition}
  Let $\{\S,\C,\Re\}$ denote a chemical reaction network.  Denote the
  complexes of the $i$th linkage class by $L_i \subset \C$.  We say a  $T \subset \C$
  consists of a  {\em union of linkage classes} if $T = \cup_{i \in I} L_i$
  for some nonempty index set $I$.
\end{definition}

\begin{definition}
  The chemical reaction network $\{\S,\C,\Re\}$ is said to be  {\em weakly
    reversible} if each linkage class of the corresponding reaction
  diagram is strongly connected.  A network is said to be
   {\em reversible} if $y' \to y \in \Re$ whenever $y \to y' \in
  \Re.$ 
  \label{def:WR}
\end{definition}

%As the results of this paper are specific to systems that are weakly reversible, it is worth pausing to point out an equivalent formulation of weak reversibility that will be used throughout. 
 It is easy to see that a chemical reaction network is weakly reversible if and only if for each reaction $y \to y'\in \Re$, there exists a sequence of complexes, $y_1,\dots, y_r\in \C$, such that $y' \to y_1 \in \Re, y_1 \to y_2 \in \Re, \cdots, y_{r-1}\to y_r\in \Re,$ and $y_r \to y\in \Re$.

\vspace{.225in}

\noindent \textbf{Dynamics.}  A chemical reaction network gives rise to a dynamical system by way of
a \textit{rate function} for each reaction.
That is, for each $y_k \to y_k'\in \Re$, or simply $k\in\{1,\dots,|\Re|\},$ we suppose the existence
of a function $\displaystyle R_k =
R_{y_k \to y_k'}$ that determines the rate of that reaction.
 The functions $R_{k}$ are
typically referred to as the \textit{kinetics} of the system and will be denoted by $\K$, or $\K(t)$ in the non-autonomous case.  The
dynamics of the system is then given by the following coupled set of
(typically nonlinear) ordinary differential equations
\begin{equation}
  \dot x(t) = \sum_{k} R_{k}(x(t),t)(y_k' - y_k),
  \label{eq:main_general}
\end{equation}
where $k$ enumerates over the reactions and $x(t) \in \R^N_{\ge 0}$ is a vector whose $i$th component represents the concentration of species $S_i$ at time $t\ge 0$.

\begin{definition}
	A chemical reaction network $\{\S,\C,\Re\}$ together with a choice of kinetics $\K$ is called a {\em chemical reaction system} and is denoted via the quadruple $\{\S,\C,\Re,\K\}$.  In the non-autonomous case where the $R_k$ can depend explicitly on $t$, we will write $\{\S,\C,\Re,\K(t)\}$.   We say that a chemical reaction system is
   {\em weakly reversible} if its underlying network is.
\end{definition}

Integrating \eqref{eq:main_general} with respect to time yields
\begin{equation*}
  x(t) = x(0) + \sum_{k} \left(\int_0^t R_k(x(s),s) ds \right)
  (y_k' - y_k). 
\end{equation*}
Therefore, $x(t) - x(0)$ remains within $S =
\text{span}\{y_k' - y_k\}_{k \in \{1,\dots,R\}}$ for all time.

\begin{definition}
  The  {\em stoichiometric subspace} of a network is the linear
  space $S = \text{\em span}\{y_k' - y_k\}_{k \in \{1,\dots,|\Re|\}}$.  The vectors $y_k' - y_k$ are called the  {\em reaction vectors}.
  \label{def:stoich_sub}
\end{definition}

Under mild conditions on the rate functions of a
system, a trajectory $x(t)$ with
strictly positive initial condition $x(0) \in \R^N_{>0}$ remains in the
strictly positive orthant $\R^N_{>0}$ for all time (see, for example, Lemma~2.1 of
\cite{Sontag2001}).  Thus, the trajectory remains in the relatively open set $(x(0)
+ S) \cap \mathbb{R}^N_{> 0}$, where $x(0) + S := \{z \in \R^N \ | \ z
= x(0) + v, \text{ for some } v \in S\}$, for all time.  In other
words, this set is \textit{forward-invariant} with respect to the
dynamics.  It is also easy to show that under the same mild conditions on $R_k$, $(x(0) + S) \cap
\mathbb{R}^N_{\ge  0}$ is forward invariant with respect to the dynamics.  The sets $(x(0) + S) \cap
\mathbb{R}^N_{> 0}$ %and the closure $(x(0) + S) \cap \mathbb{R}^N_{\ge  0}$ are typically referred to as the \textit{positive} and \textit{non-negative stoichiometric compatibility classes}, respectively.
will be referred to as the \textit{positive stoichiometric compatibility classes}, or simply as the \textit{positive classes}.

The most common kinetics is that of \textit{mass-action kinetics}. A
chemical reaction system is said to have mass-action kinetics if all rate functions $R_{k} = R_{y_k \to y_k'}$ 
take the multiplicative form
\begin{equation}
  R_{k}(x) =  \kappa_k x_1^{y_{k,1}} x_2^{y_{k,2}} \cdots x_N^{y_{k,N}},
  \label{eq:massaction}
\end{equation}
where $\kappa_k$ is a positive reaction rate constant and $y_k$ is the source complex for the reaction.  For $u\in \R^N_{\ge 0}$ and $v \in \R^N$, we define 
\begin{equation*}
 u^v \eqdef u_1^{v_1} \cdots u_N^{v_N},
\end{equation*}
 where we have
adopted the convention that $0^0 = 1$, and the above is undefined if $u_i = 0$ when $v_i < 0$.  Mass action kinetics can then be written succinctly as $R_k(x) = \kappa_k x^{y_k}.$
Combining \eqref{eq:main_general} and \eqref{eq:massaction} gives the
following system of differential equations, which is the main object of study in this paper,
\begin{equation}
  \dot x(t) = \sum_{k} \kappa_k x(t)^{y_k}(y_k' - y_k).
  \label{eq:main}
\end{equation}

While it is the properties of solutions to the system \eqref{eq:main} that are of most interest to us, it will be natural for us to consider systems with a generalized form of mass-action kinetics.  The following definition is similar to Definition 2.6 in \cite{CraciunPantea} for  ``$\kappa$-variable mass-action systems.''  See also \cite{Angeli2011} for a recent treatment of chemical reaction systems with non-autonomous dynamics.

\begin{definition}
	We say that the non-autonomous system $\{\S,\C,\Re,\K(t)\}$ has  {\em bounded mass-action kinetics} if there exists an $\eta > 0$ such that for each $k\in \{1,\dots,|\Re|\}$
	\begin{equation*}
		R_k(x,t) = \kappa_k(t) x^{y_k},
	\end{equation*}
	where $\eta < \kappa_k(t) < 1/\eta$ for all  $t \ge 0$ and $k \in \{1,\dots,|\Re|\}$.
\end{definition}

\subsection{Complex balanced equilibria and the deficiency of a network}
\label{sec:conjectures}

The Global Attractor Conjecture, which was stated in Section \ref{sec:background}, is concerned with the asymptotic 
stability of  complex-balanced equilibria for mass-action systems. For each complex $y 
\in \C$ we will write $\{k \ | \ y_k = y\}$ and $\{k \ | \ y_k' = 
y\}$ for the subsets of reactions $y_k\to y_k' \in \mathcal{R}$ for which 
$y$ is the source and product complex, respectively.

\begin{definition} 
  We say $c$ is an {\em equilibrium} of the dynamical system $\dot x(t) = f(x(t))$, if $f(c) = 0$.  An equilibrium $c \in \R^N_{\ge 0}$ of \eqref{eq:main} is 
  said to be {\em complex-balanced} if the following equality  
  holds for each complex $y \in \C$: 
  \begin{equation*} 
    \sum_{\{k \ | \ y_k = y \}} \kappa_k   c^{y_k} 
    ~=~ \sum_{\{k \ | \ y_k' = y\}}  \kappa_k c
    ^{y_k}.  
    \label{def:cb}
  \end{equation*} 
 \end{definition} 
 
  Note that on the right hand side of the above equality, $y_k$ represents the source complex for a given reaction for which $y$ is the product complex, whereas on the left hand side each source complex is identically equal to $y$.
  Thus, $c$ is a complex-balanced equilibrium if for all complexes $y \in \C$, at 
  concentration $c$ the 
  sum of rates for reactions for which  $y$ is the source is equal to 
  the sum of rates for reactions for which  $y$ is the product.  That is, crudely, the total flux into complex $y$ is equal to the total flux out of complex $y$.
  A {\em complex-balanced system} is a mass-action system 
  $\{\S,\C,\Re,\K\}$ that admits a  complex-balanced 
  equilibrium with strictly positive components. 
  % need the assumption of strictly positive, because the presence of 
  % boundary c-bal equilibria does not imply the same for the interior 
 
In \cite{CraciunShiu09}, complex-balanced systems are called ``toric dynamical 
systems'' in order to highlight their inherent algebraic structure. 
 There are two 
important special cases of complex-balanced systems: the 
detailed-balanced systems and the zero deficiency systems. 
 
% :Definition of d-balancing 
\begin{definition} 
  An equilibrium $c \in \R^N_{\ge 0}$ of a reversible system 
  with dynamics given by mass-action kinetics \eqref{eq:main} is said to be 
  {\em detailed-balanced} if for any pair of reversible reactions 
  $y_k \rightleftarrows y_k'$ with forward reaction rate $\kappa_k$ 
  and backward rate $\kappa_k'$ the following equality holds: $\kappa_k c^{y_k} = \kappa_k' c^{y_k'}.$ 
  \end{definition}
  
  %That is, $c$ is a detailed-balanced equilibrium if the 
  %forward rate of each reaction equals the reverse rate at 
  %concentration $c$.  
  A \textit{detailed-balanced system} is a 
  reversible system with dynamics given by mass-action kinetics \eqref{eq:main} that admits 
  a strictly positive detailed-balanced equilibrium. 
It is clear that detailed-balancing implies complex-balancing. 
 
\begin{definition} 
  For a chemical reaction network 
  $\{\mathcal{S},\mathcal{C},\mathcal{R}\}$, let $n$ denote the number 
  of complexes, $l$ the number of linkage classes, and $s$ the 
  dimension of the stoichiometric subspace, $S$.  The 
  {\em deficiency} of the reaction network is the integer $n-l-s$. 
  \label{def:deficiency}
\end{definition} 
 
The deficiency of a reaction network is non-negative 
because it can be interpreted as either the dimension of a 
certain linear subspace \cite{FeinbergLec79} or the codimension of a certain 
ideal \cite{CraciunShiu09}.  Note that the deficiency depends only on the 
reaction network and not the choice of kinetics. The Deficiency Zero Theorem  tells us that any weakly 
reversible reaction system \eqref{eq:main} whose deficiency is zero 
is complex-balanced, and that this fact is independent of the choice 
of rate constants $\kappa_k$ 
\cite{FeinbergLec79}.   
% In fact, this property of being c-bal regardless of rate parameters 
% defines the space of deficiency zero systems. 
On the other hand, a reaction diagram with a deficiency that is 
positive may give rise to a system that is both complex- and 
detailed-balanced, complex- but not detailed-balanced, or neither, 
depending on the values of the rate constants $\kappa_{k}$ 
\cite{CraciunShiu09,Feinberg72,Feinberg89,Horn72}. 
 
Complex-balanced systems have the property that there is a unique, complex-balanced equilibrium within the interior of each positive stoichiometric compatibility class \cite{Horn72, Horn74, HornJack72}.  Thus, proving that each such equilibrium is globally asymptotically stable relative to its positive class, i.e. showing the conclusion of the Global Attractor Conjecture holds, would completely characterize the long-time behavior of these systems.

%
%\vspace{.125in}

%\noindent \textbf{Notation:}

\section{Projected dynamical systems and reduced reaction networks}
\label{sec:projection}

The two related concepts  presented below, projected dynamics and reduced reaction networks, will be used in the proofs of the main theorems of Section \ref{sec:results}.  There, we will  consider a bounded trajectory of a system  and incorporate the dynamics of those species which do not approach the boundary of the positive class into the rate constants, thereby yielding a system with bounded mass-action kinetics.  This will be the {\em projection} of the dynamics onto those species that do approach the boundary.  The resulting reaction network will be the \textit{reduced network}.

\subsection{Projected dynamics}
\label{sec:proj_dyn}

As discussed in Section \ref{sec:def_concepts}, our  interest is in the qualitative dynamics of an $N$ dimensional,  autonomous systems of differential equations with parameters $\ka = (\kappa_1,\dots,\kappa_{|\Re|})$.  That is, we are considering systems of the general form
\begin{align}
\begin{split}
	\dot x_1(t) &= f_1(\ka, x_1(t),\dots,x_N(t))\\
	&\hspace{.07in} \vdots\\
	\dot x_N(t) &= f_N(\ka, x_1(t),\dots,x_N(t)).
	\end{split}
	\label{eq:gen1}
\end{align}

To study these systems, it will be natural to later consider an associated \textit{non-autonomous} set of differential equations constructed by projecting \eqref{eq:gen1} onto a subset of the dependent variables.   Specifically, let  $U \subset \{1,\dots,N\}$  be nonempty with $|U| = M \le N$.  Without loss of generality, assume for now that $U = \{1,\dots,M\}$. We now consider the dynamical system in $\R^M$ with state vector $x|_U$, whose $i$th component, denoted $x|_{U,i}$, is for all time equal to $x_i$, defined in \eqref{eq:gen1}. We see that the vector valued function $x|_U$ satisfies
\begin{align}
\begin{split}
	\dot x|_{U,1}(t) &= f_1(\ka, x|_{U,1}(t),\dots,x|_{U,M}(t),x_{M+1}(t),\dots , x_N(t))\\
	&\hspace{.3in} \eqdef \hat f_1(\hat \kappa(t),  x|_{U,1}(t),\dots,x|_{U,M}(t))\\
	&\hspace{.07in} \vdots\\
	\dot x|_{U,M}(t) &= f_M(\ka, x|_{U,1}(t),\dots,x|_{U,M}(t),x_{M+1}(t),\dots , x_N(t))\\
	&\hspace{.3in} \eqdef \hat f_M(\hat \kappa(t),  x|_{U,1}(t),\dots,x|_{U,M}(t)),
	\end{split}
	\label{eq:proj_gen}
\end{align}
where $\hat \kappa(t) = (\kappa, x_{M+1}(t),\dots , x_N(t))$ and $\hat f_i$, $i \in \{1,\dots, M\}$, are defined via the above equalities.  Thus, the system $x|_{U,i}$ satisfies the \textit{same} differential equations as does the system for $x_i$, except any dependence upon the variables $x_j$, for $j \ge M+1$, has been incorporated into the dynamics as known, time-dependent functions.  That is, they are now viewed as inputs, or perhaps forcing, to the system.
 We will call the system \eqref{eq:proj_gen} the \textit{projected dynamics} of \eqref{eq:gen1} with respect to $U$.
%
%Note that the dependencies of the $f_j$ have changed in the transition from  \eqref{eq:gen1} to \eqref{eq:proj_gen}.  That is, in \eqref{eq:proj_gen} we group the elements of $U$ and $U^c$ together.  This should not cause confusion.  

For example, consider the system 
\begin{align}
\begin{split}
	\dot x_1 &=- \kappa_1 x_1x_2^2  - \kappa_2 x_1x_3  + \kappa_5 x_2\\
	\dot x_2 &=\kappa_3 x_3 -2\kappa_1 x_1x_2^2 - \kappa_5x_2\\
	\dot x_3 &= \kappa_4 +\kappa_1x_1x_2^2 -\kappa_3x_3
	\end{split},
	\label{ex:project}
\end{align}
where $\kappa = (\kappa_1,\dots,\kappa_5)\in \R^5_{>0}$.  Then for $U = \{1,3\}$ the projected dynamics of \eqref{ex:project} with respect to $U$ is
\begin{align}
\begin{split}
	\dot x_1 &=- \kappa_1 \zeta_{1}(t) x_1  - \kappa_2 x_1x_3  + \kappa_5 \zeta_{2}(t)\\
	\dot x_3 &= \kappa_4 +\kappa_1 \zeta_{1}(t) x_1-\kappa_3x_3
	\end{split}
	\label{eq:proj2}
	\end{align}
where $\zeta_{1}(t) = x_2(t)^2,\ \zeta_2(t) = x_2(t),$ and $x_2(t)$ is still defined via the system \eqref{ex:project}.  The goal would now be to translate any control we can get over $x_2(t)$, and hence $\zeta(t)$, into qualitative information about the behavior of $x_1$ and $x_3$.  Later, we will simply incorporate the function $\zeta$ of  \eqref{eq:proj2} into the variables $\kappa_k$ and view each $\kappa_k(t)$ as a function of time.

\subsection{Reduced reaction networks}
\label{sec:reduced_networks}

We begin with more notation.  Let  $v \in \R^N$ for some $N \ge 1$, and let $U \subset \{1,\dots,N\}$ be nonempty.  We write $U[j]$ for the $j$th component of $U$. We then write $v|_U$ to denote the vector of size $|U|$ with 
\begin{equation*}
 v|_{U,j} = (v|_U)_j \eqdef v_{U[j]}
\end{equation*}
 for $j \in \{1,\dots, |U|\}.$
  Thus, $v|_U$ simply denotes the projection of $v$ onto the components enumerated by $U$.
For example, if $N = 8$ and $U = \{2,4,7\}$, then for any $v \in \R^8$, $v|_{U} = (v_2,v_4,v_7)$.

\begin{definition}
  Consider a reaction network $\{\S,\C,\Re\}$ with $\S = \{S_1,\dots, S_N\}$ and let $U \subset
  \{1,\dots, N\}$ be nonempty.  The  {\em reduced reaction network of $\{\S,\C,\Re\}$
  associated with $U$} is the reaction network $\{\S_U, \C_U, \Re_U\}$
  constructed in the following manner:
  \begin{enumerate}
  \item Set $\S_U = \{S_i \in \S\ : \ i \in U\}$.
  \item Set $\C_U = \{y|_U \ : \ y\in  \C\}$.  We say the complex $y$
    {\em reduces to} the complex $y|_U$.
  \item Set $\Re_U = \{y|_U \to y|_U' \ : \ y \to y' \in \Re, \text{ and for which $y|_U \ne y'|_U$}\}$. 
  \item If a resulting linkage class consists of a single complex, we
    delete that complex from $\C_U$.
  \end{enumerate}
  \label{def:reduced}
\end{definition}

\begin{example}
Consider the reaction network with species $\S = \{S_1,\dots,S_5\}$ and reaction diagram
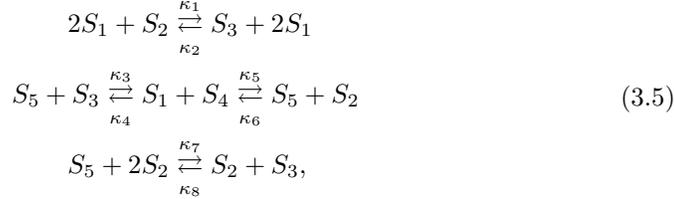
\begin{align}
\begin{split}
    2S_1 + S_2 &\overset{\kappa_1}{\underset{\kappa_2}{\rightleftarrows}} S_3 + 2S_1\\
    S_5 + S_3\overset{\kappa_3}{\underset{\kappa_4}{ \rightleftarrows}} S_1 &+ S_4 \overset{\kappa_5}{\underset{\kappa_6}{\rightleftarrows}} S_5 + S_2\\
    S_5 + 2S_2 &\overset{\kappa_7}{\underset{\kappa_8}{\rightleftarrows}} S_2 + S_3,
    \end{split}
    \label{eq:ex_diagram}
\end{align}
where we have ordered the reactions and incorporated the rate constants into the reaction diagram.  Let $U = \{1,4,5\}$.  Then, $S_U = \{S_1,S_4,S_5\}$, $\C_U = \{S_5, S_1 + S_4, \vec 0\}$, and the resulting diagram of the reduced reaction network is
\begin{align}
	\vec 0 \rightleftarrows S_5 \rightleftarrows S_1 + S_4.
	\label{eq:ex_diagram_reduced}
\end{align}
Here, both the complex $2S_1+S_2$ and $S_3 + 2S_1$ reduce to $2S_1$. However, by rule 3 in Definition \ref{def:reduced} we do not include $2S_1 \to 2S_1$ in  $\Re_U$, and by rule 4 we delete $2S_1$ from $\C_U$.  Note also that the original network has three linkage classes whereas the reduced network  has only one.
\end{example}

Note that because of rule 4 in Definition
\ref{def:reduced}, it is possible to have $S_i \in \S_U$, even though
$S_i$ does not appear in any complex in $\C_U$.  For example, if $1
\in U$, but $2,3 \notin U$, and the only reactions in which $S_1$
participates are $S_1 + S_2 \rightleftarrows S_1 + S_3$, then $S_1 \in
\S_U$, even though $S_1$ does not appear in any complex in $\C_U$.  In this case, the reduced reaction network has inactive species, see Definition \ref{def:crn}.
Note, however, that this situation only arises if the concentration
of $S_1$ was time independent in the original system.  Thus, the original system could have
been reduced by incorporating $S_1$ into the rate constants.    Such an incorporation can have the effect of lowering the deficiency of the network without changing the dynamics (see \cite{CraciunPantea2008} to see a treatment of how different network structures may give rise to the same dynamical system).

%We therefore make the following standing assumption pertaining to any chemical reaction network under consideration:

%\vspace{.125in}

%\noindent \textbf{Standing Assumption:}  For a chemical reaction network $\{\S,\C,\Re\}$, for each $S_i \in \S$ there is a $k \in \{1,\dots,R\}$ for which $(y_k' - y_k)_i \ne 0$.

%\vspace{.125in}

%So long as $\{\S,\C,\Re\}$ satisfies the above standing assumption, it is clear that for any $U\subset \{1,\dots,N\}$, the reduced network $\{\S_U,\C_U,\Re_U\}$ satisfies the requirements of Definition \ref{def:crn} and is therefore itself a reaction network.

The following lemmas  give some insight into how the structure of $\{\S_U,\C_U, \Re_U\}$ depends upon the structure of $\{\S,\C,\Re\}$.  Both will be used in Section \ref{sec:results} in the proof of our main results.

\begin{lemma}
	Let $\{\S,\C,\Re\}$ be a reaction network with $\S = \{S_1,\dots, S_N\}$.  Let $U \subset \{1,\dots,N\}$ be nonempty.  Then, the reduced reaction network $\{\S_U,\C_U,\Re_U\}$ has less than or equal to the number of linkage classes as $\{\S,\C,\Re\}$.
	\label{lem:linkage}
\end{lemma}

\begin{proof}
	%The above example demonstrates the possibility of having fewer linkage classes.  All that is required to complete the proof, therefore, is to argue that $\{\S_U,\C_U,\Re_U\}$ can not have \textit{more} linkage classes than $\{\S,\C,\Re\}$.  However, 
	Condition 3 of Definition \ref{def:reduced} shows that if $y_1, y_2\in \C$ are in the same linkage class, then $y_1|_U$ and $y_2|_U$ are also.  Thus the result is shown simply by counting the number of unique linkage classes of  $\{\S,\C,\Re\}$ and $\{\S_U,\C_U,\Re_U\}$ by enumerating over the complexes $\C$ and $\C_U$, respectively.
\end{proof}

\begin{lemma}
  Suppose that $\{\S,\C,\Re\}$, with $\S = \{S_1,\dots, S_N\}$, is weakly reversible and that $U \subset \{1,\dots, N\}$ is nonempty.  Then $\{\S_U, \C_U, \Re_U\}$ is weakly reversible.
  \label{lemma:WR}
\end{lemma}
\begin{proof}
  Suppose $y|_U \to y|_U'\in \Re$.  By construction there are
  complexes $y,y' \in \C$, with $y\to y'\in \Re$, that reduce to $y|_U, y|_U'$.  By the weak reversibility of $\{\S,\C,\Re\}$, there
  is a sequence of directed reactions, $y_k \to y_k' \in \Re$, beginning with $y'$ and
  ending with $y$.  If for each $y_k \to y_k'$ of this sequence we have $y_k|_U \ne y_k'|_U$, then, by construction, there is a
  sequence of directed reactions in $\Re_U$ beginning with $y|_U'$
  and ending with $y|_U$. If $y_k|_U = y_k'|_U$ for one of the $k$, such a sequence still exists except now $y_{k}|_U \to y_{k+1}'|_U$. Similar reasoning  holds when $y_k|_U = y_k'|_U$ for more than one $k$.
\end{proof}

\vspace{.1in}

We need to provide the reduced network $\{\S_U,\C_U,\Re_U\}$ with a natural kinetics.  The kinetics, $\K(t)$, is given via the projection of the dynamics, described in Section \ref{sec:proj_dyn}, of $\{\S,\C,\Re,\K\}$  onto the elements of $U$.  The variables $\kappa_k(t)$ are now functions of time.  Note that the dynamics of the resulting projected system depends upon the dynamics of the original system.

\begin{example}
Again consider the reaction system with species $\S = \{S_1,\dots,S_5\}$ and reaction diagram \eqref{eq:ex_diagram}.  For $U = \{1,4,5\}$, the reduced network was \eqref{eq:ex_diagram_reduced}.  Incorporating the projected dynamics yields
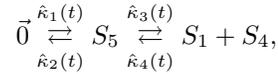
\begin{align*}
	\vec 0 \overset{\hat \kappa_1(t)}{\underset{\hat \kappa_2(t)}{\rightleftarrows}} S_5 \overset{\hat \kappa_3(t)}{\underset{\hat \kappa_4(t)}{\rightleftarrows}} S_1 + S_4,
\end{align*}
where $\hat \kappa_1(t) = \kappa_8 x_2(t)x_3(t), \ \hat \kappa_2(t) = \kappa_7 x_2(t)^2,\  \hat \kappa_3(t) = \kappa_3x_3(t) + \kappa_6 x_2(t),\ \hat \kappa_4(t) = (\kappa_4 + \kappa_5),$ and the indexing $k$ of the $\hat \kappa_k(t)$ is over reactions in the reduced network and not the original network.
\end{example}

It is important to note that the variables $\hat \kappa_k(t)$ for the reduced system, which take the place of the rate constants, are always non-negative as they consist of positive linear combinations of non-negative monomials of the variables $x_j$ for which $j \notin U$, see \eqref{eq:new_rates} below.  Also, while the reduced system is a non-autonomous system with generalized mass-action kinetics, the functions $\hat \kappa_k(t)$ are not necessarily bounded either above or below. Finally, note that the functions $\hat \kappa_k(t)$ depend explicitly on the original \textit{trajectory} of the system and, in particular, will be different functions of time for different initial conditions of the system $\{\S,\C,\Re,\K\}$.

It is useful to have a more formal representation of the projected dynamics.  Thus,  let $\{\S,\C,\Re, \K\}$ be a reaction network with mass-action
  kinetics, $\K = \{\kappa_k\}$.  Let $U \subset
  \{1,\dots, |\S|\}$ be nonempty. The reduced mass-action system of $\{\S,\C,\Re,
  \K\}$ with respect to $U$ is the non-autonomous mass-action system
  $\{\S_U, \C_U, \Re_U, \K_U(t)\}$, with $\K_U(t) = \{\kappa_k(t)\}$, where the indexing $k$ of the $ \kappa_k(t)$  is over reactions in the reduced network and not the original network, and where for $y_k|_U \to  y_k'|_U \in \Re_U$
  \begin{equation}
    \kappa_k(t) = \sum_{\{z_i \to z_i' \in \Re \ : \  y_k|_U =
      z_i|_U \text{ and } y_k'|_U = z_i'|_U\}} \kappa_i
    \left(x(t)|_{U^c}\right)^{z_i|_{U^c}},
    \label{eq:new_rates}
  \end{equation}
  where  $x(t)$ is the solution to equation \eqref{eq:main} for the system $\{\S,\C,\Re,\K\}$.

We reiterate the fact  that based upon the above definitions, the differential equations
governing the dynamics of $x_i$ for $i \in U$ for the reduced
system  are exactly the
same as the differential equations for $x_i$ for $i \in U$ of the
original mass-action system.  In Section \ref{sec:results}, we will project the dynamics of a bounded trajectory onto the subset of the species that go to the boundary, thereby producing a trajectory that satisfies a system with bounded mass-action kinetics.  Note that  both the resulting reduced network and the bounds on the kinetics will depend upon the initial condition of the original system.

\section{Main results}
\label{sec:results}

We begin in Section \ref{sec:partition} by introducing the concept of partitioning a set of vectors along a sequence, which will be one of our main analytical tools and will allow us to group the relevant monomials of the dynamical system along sequences of trajectory points into classes with comparable growths.   In Section \ref{sec:results2} our main results will be stated and proved.  The main results will focus on non-autonomous systems because we will later project our system of interest, which is autonomous, to the non-autonomous system consisting of those species that approach the boundary along a subsequence of times.

\subsection{Partitioning vectors along a sequence}
\label{sec:partition}
We begin by recalling that for any vectors $u,v$ such that $u \in \R^N_{\ge 0}$ and $v\in \R^N$ we define $u^v \eqdef u_1^{v_1}\cdots u_N^{v_N}$, where we use the convention $0^0 = 1$.  All sequences and subsequences will be indexed by non-negative integers.  

\begin{definition}
  Let $\C$ denote a finite set of vectors in $\R^N$. Let $x_n \in
  \R^N_{> 0}$ denote a sequence of points in the strictly positive orthant.  We say that $\C$ is
   {\em partitioned along the sequence} $\{x_n\}$ if there exists $T_i \subset \C$, $i = 1, \dots,P$, termed \textit{tiers}, such that
    $T_i \ne \emptyset$, $T_i \cap T_j = \emptyset$ if $i \ne j$, and
    $\cup_i T_i = \C$ (that is, the tiers constitute a {\em partition} of
    $\C$), and a constant $C > 1$,  such that
  \begin{enumerate}[$(i)$]
  \item if $y_j,y_k \in T_i$ for some $i \in \{1,\dots,P\}$, then for all $n$
    \begin{equation*}
      \frac{1}{C} x_n^{y_j} \le x_n^{y_k} \le C x_n^{y_j},
    \end{equation*}
    which is equivalent to either of the conditions
    \begin{equation*}
    	  \frac{1}{C} \le       \frac{x_n^{y_k}}{x_n^{y_j}} \le C    \quad  \text{or}  \quad \frac{1}{C} \le  x_n^{y_k-y_j} \le C.
    \end{equation*}

  \item if $y_k \in T_i$ and $y_j \in T_{i+m}$ for some $m \in \{1,\dots,P-i\}$,
    then
    \begin{equation*}
      \frac{x_n^{y_k}}{x_n^{y_j}}  \to \infty, \quad
      \text{as } n \to \infty.
    \end{equation*}
  \end{enumerate}
  % In the above definition of $f_{j,k}$, if the complexes under
  % consideration come from a single reaction $y_k \to y_k'$, we will
  % often simply write $f_k = x^{y_k'}/x^{y_k}$.
  \label{def:partition}
\end{definition}
Therefore, we have a natural ordering of the tiers: $T_1 \succ T_2 \succ T_3 \succ \cdots \succ T_P$, and we say $T_1$ is the ``highest'' tier, whereas $T_P$ is the ``lowest'' tier.  

Throughout the paper, Definition \ref{def:partition} will be used in the context of the sequence $\{x_n\}$ being trajectory points, and the set of vectors $\C$ being the complex vectors.

%\begin{definition}
%  If $\{T_i\}$ is a partition of a set of vectors $\C$ such that each
%  $T_i$ consists of a single vector, then we say the partition is
%   {\em trivial}.
%\end{definition}

The following lemma, which is critical for our purposes, states that given a set of vectors and a sequence of points in $\R^N_{>0}$, there always exists a subsequence along which the vectors are partitioned.

\begin{lemma}
  Let $\C$ denote a finite set of vectors in $\R^N$.  Let $x_n$ be a
  sequence of points in $\R^N_{>0}$.  Then, there exists a subsequence of
  $\{x_n\}$ along which $\C$ is partitioned.
  \label{lem:partition}
\end{lemma}
\begin{proof}
  Denote the elements of $\C$ as $y_i$, $1\le i \le r$, where $|\C| = r$.  Note that
  there are $r! < \infty$ ways to order the elements of $\C$.
  Therefore, there exists a reordering of the vectors in $\C$ such that the set of indices $n$ for which
  \begin{equation}
     x_n^{y_1} \ge x_n^{y_2} \ge \cdots \ge x_n^{y_r}
     \label{eq:ordering}
  \end{equation}
  holds is infinite; letting $n_k$ denote those indices, we obtain the corresponding subsequence $\{x_{n_k}\}$.
  Thus, we have instituted an ordering, $y_1 \succ y_2 \succ \cdots \succ y_r$ along this subsequence, and $y_1$ can be viewed as maximal. The goal now is to simply get more information about this ordering (along further subsequences) and ask which vectors stay ``close'' to each other, and which diverge.  This will give us the natural dividing lines for our tiers.
  
  For $i \in  \{1,\dots,r-1\}$, define $\psi_{i}:\R^N_{>0} \to \R$ by $\psi_{i}(x)\eqdef x^{y_i} / x^{y_{i+1}}.$   By the inequalities \eqref{eq:ordering}, for each $i\in \{1,\dots,r-1\}$ and for $k\ge 1$, we have $\psi_i(x_{n_k}) \ge 1$.  We will construct the tiers.  We begin by setting $T_1 = \{y_1\}$.  Next, we ask if 
  \begin{equation}
      \liminf_{k\to \infty} \psi_{1}(x_{n_{k}}) < \infty.
      \label{eq:lim_inf}
  \end{equation}
  If \eqref{eq:lim_inf} holds, we set $T_1 = \{y_1,y_2\}$ and redefine our sequence
  $\{x_{n_k}\}$ as an appropriate subsequence so that $\lim_{k\to \infty} \psi_{1}(x_{n_{k}})$ exists, and is finite.  Next, we ask if  
  \begin{equation*}
    \liminf_{k\to \infty} \psi_{2}(x_{n_{k}}) < \infty
   \end{equation*}
    along this new subsequence.  If so, we set $T_1 = \{y_1,y_2,y_3\}$ and redefine our sequence appropriately so that $\lim_{k\to \infty} \psi_{2}(x_{n_{k}})$ exists and is finite.  This process will either terminate with $T_1 = \C$ or when $\liminf_{k \to \infty} \psi_{b-1}(x_{n_k})  = \infty$ for some $b \in \{2,\dots, r\}$.  If such a $b$ exists, we have $T_1 = \{y_1,\dots,y_{b-1}\}$, and then we begin building the second tier by setting $T_2 = \{y_b\}$.    Now fill $T_2$ in the same
  manner that we did $T_1$ by looking at the values of
  $\liminf_{k\to \infty}\psi_i(x_{n_{k}})$ for the appropriate $i$'s.
  Repeat this process, always redefining the sequence as the
  subsequence guaranteed to exist at each step, until we have a
  sequence of tiers $T_1, T_2, \dots, T_P$.   
  
  It remains to find the appropriate $C>1$ for Definition \ref{def:partition}.  For each pair $y_j,y_{j+1} \in T_i$, we define $C_{i,j} \eqdef  \lim_{k \to \infty} x_{n_k}^{y_j}/x_{n_{k}}^{y_{j+1}}$, which exists and is finite by construction.   For each such $i,j$, let $\widetilde C_{i,j} = C_{i,j} + 1$ and note that there is a $K \ge 1$ so that
  \begin{align*}
  	1\le \frac{x_{n_k}^{y_j}}{x_{n_k}^{y_{j+1}}} \le \widetilde C_{i,j},
  \end{align*}
  for all $k \ge K$, and all relevant $i,j$ (that is, this bound is uniform in $i$ and $j$).  Let 
    \begin{equation*}
  	C \eqdef \max_{i \in \{1,\dots,P\}} \prod_{\{ j\ :\ y_j,y_{j+1}\in T_i\}}\widetilde C_{i,j}.
  \end{equation*}
  Then, restricting ourselves to the subsequence with $k\ge K$, 
  $\C$ is
  now partitioned along the resulting subsequence with tiers $\{T_i\}$, $i = 1,\dots,P$, and constant $C > 1$.
\end{proof}

The following lemma states that for any set of vectors in $\R^m$, either their span includes a non-zero vector in the non-positive orthant $\R^m_{\le 0}$, or there is vector normal to their span that intersects the strictly positive orthant.
\begin{lemma}[Stiemke's Theorem, \cite{Stiemke1915}]
  For $i = 1,\dots, n$, let $u_i \in \R^m$. Either there exists an $\alpha \in \R^n$ such that 
  \begin{equation*}
    \left(\sum_{i=1}^n \alpha_i u_i\right)_j \le  0, \quad j = 1,\dots,m
  \end{equation*}
  and such that at least one of the inequalities is strict,
  or there is a $w \in \R^m_{> 0}$ such that $w\cdot u_i=0$ for each
  $i\in \{1,\dots, n\}$.
  \label{lem:Stiemke}
\end{lemma}
%

% \begin{lemma}[Gordan's Theorem, \cite{Gordan1873}]
%   Let $u_i \in \R^N$ be a set of $m$ vectors. Either the set of
%   inequalities
%   \begin{equation*}
%     \left(\sum_{i=1}^m \alpha_i u_i\right)_j <  0, \quad j = 1,\dots,N
%   \end{equation*}
%   has a solution $\alpha \in \R^m$, or there is nonzero $w \in
%   \R^N_{\ge 0}$ such that $w\cdot u_i=0$ for each $i$.
%   \label{lem:Gordan}
% \end{lemma}

\begin{definition}
	Let $w \in \R^N$.  The set $\{i \in \{1,\dots,N\} \ : \ w_i \ne 0\}$ is called the  {\em support} of $w$.
\label{def:support}
\end{definition}

\begin{definition}
  Let $\C$ denote a finite set of vectors in $\R^N$. Let $\{T_i\}$
  denote a partition of $\C$.  Let $U \subset \{1,\dots,
  N\}$ be nonempty.  We say that the vector $w \in \R^N_{\ge 0}$ is a  {\em non-negative
  conservation relation that respects the pair} $(U,\{T_i\})$ if the
  following two conditions hold:
  \begin{enumerate}
  \item $w_i > 0$ if and only if $i \in U$.  That is, the \textit{support} of $w$ is $U$.
  \item Whenever $y_j,y_{\ell} \in T_i$
    for some $i$, we have that $w \cdot (y_j - y_{\ell}) = 0$.
  \end{enumerate}
\end{definition}

%In the above definition we may or may not have that $w\cdot (y_j - y_{\ell}) = 0$ if $y_j\in T_{i_1}$ and $y_{\ell}\in T_{i_2}$ for $i_1 \ne i_2$.  Also, 
Note that if each $T_i$ consists of a single element, then any vector $w$ whose support is $U$ satisfies the requirements of the definition as the second condition holds automatically. Also note that if $\C = T_1$, then the type of conservation relation described above is the usual concept in chemical reaction network theory.

If in the following theorem, $\C$ is taken to be the set of complexes of a reaction network, then the theorem guarantees that there must be a conservation relation among certain subsets of the reactions if a trajectory converges to the boundary of the positive orthant.

%
%\begin{theorem}
%  Let $\C$ denote a finite set of vectors in $\R^N$.  Let $x_n \in  \R^N_{>0}$ denote a sequence of points such that:
%  \begin{enumerate}
%  \item There is a $K>0$ such that $|x_n| \le K$ for all $n$,
%  \item $x_n \to \partial \R^N_{\ge 0}$ in that $\text{\em dist}(x_n, \partial \R^N_{\ge 0}) \to 0$, as $n \to \infty$.
%  % following sense: there is a sequence  $\epsilon_n\to 0$, as $n \to \infty$ such that for each $n\ge 1$, we have $x_{n,i} \le \epsilon_n$ for at  least one $i\in \{1,\dots,N\}$.
%  \end{enumerate} 
%  Let $\{x_{{n_k}}\}$ be any convergent subsequence of the sequence
%  with limit point $z$, say.  Let $U = U(z) =  \{i \in \{1,\dots,N\} \, : \,
%  z_i = 0\}$.  Finally, suppose that $\C$ is partitioned along
%  $\{x_{n_k}\}$ with tiers and constants $T_i, C_i$, for $i=1,\dots,
%  P$, respectively.  Then,  there is a
%  non-negative conservation relation $w \in \R^N_{\ge 0}$ that
%  respects the pair $(U,\{T_i\})$.
%  \label{thm:conservation}
%\end{theorem}

\begin{theorem}
  Let $\C$ denote a finite set of vectors in $\R^N$.  Let $x_n \in  \R^N_{>0}$ denote a sequence of points with $x_n \to z \in \partial \R^N_{\ge 0}$, as $n \to \infty$.
 Let $U = U(z) =  \{i \in \{1,\dots,N\} \, : \,
  z_i = 0\}$.  Finally, suppose that $\C$ is partitioned along
  $\{x_{n}\}$ with tiers $T_i$, for $i=1,\dots,
  P$, and constant $C>0$.  Then,  there is a
  non-negative conservation relation $w \in \R^N_{\ge 0}$ that
  respects the pair $(U,\{T_i\})$.
  \label{thm:conservation}
\end{theorem}

\begin{proof}
  %Note that by condition 1., there is such a convergent subsequence and by condition 2., $U$ is nonempty.  
  We suppose, in order to find a contradiction, that there is no non-negative
  conservation relation that respects the pair $(U,\{T_i\})$.
  % Let $m = |U|$ be the number of elements in $U$. For each $y_k \in
  % C$, we let $\overline y_k = y_k|_U$, and then define $\overline \C
  % = \{\overline y_k\}$ and $\overline x_n = x_n|_U$.  For $i \in
  % \{1,\dots,P\}$, define
  % \begin{equation*}
  %   \overline T_i = \{v \in \R^m\, : \, v = y_k|_U \text{ for some }
  %   y_k \in T_i\}.
  % \end{equation*}
  % Note (maybe prove), that $\overline \C$ is partitioned along
  % $\overline x_n$ with tiers $\{\overline T_i\}_{i = 1}^P$ (note we
  % must have the same number of tiers).  Note (again, should pull
  % this out) that the partition $\{\overline T_i\}$ is non-trivial,
  % actually, I think it could be.  Which is good.
  Define the sets $W_i\subset \R^N$, for $i = 1,\dots P,$ and $W \subset \R^N$  via
  \begin{equation*}
    W_i \eqdef   \{y_{j} - y_{k}\ | \
    y_{j},
    y_{k} \in T_i \} , \qquad W \eqdef \bigcup_{i = 1}^{P} W_i,
  \end{equation*}
  and denote the elements of $W$ by $\{u_k\}$.  Note that if $T_i$ consists of a single element, then $W_i$ consists solely of the vector $\vec 0$.  Let $m = |U| >0$ be the number of elements
  in $U$ and let $W_i|_U \subset \R^m$  and $W|_U \subset \R^m$
   be the restrictions of $W_i$ and $W$ to the
  components associated with the index set $U$, as discussed in Section \ref{sec:reduced_networks}.  Denote the elements of $W|_U$ by
  $\{v_k\}$.  Thus, collecting terminology, $u_k \in \R^N$, whereas $v_k\in \R^m$, and for each $u_k\in W$, there is a
  corresponding $v_k \in W|_U$ for which $u_k|_U = v_k$, however the  mapping $\cdot |_U$ need not be injective.
  
  The set $W|_U$ must contain at least one nonzero vector because otherwise any non-negative vector with support $U$ would be a non-negative conservation relation that respects the pair $(U,\{T_i\})$, but we have assumed that no such relation exists.

  Because we have assumed
  there is no conservation relation that respects the pair
  $(U,\{T_i\})$, we may conclude by Lemma \ref{lem:Stiemke} that
  $\text{span}(W|_U)$ must intersect $\R^m_{\le 0}$ in a
  non-trivial manner.  That is, there exist $c_{k} \in \R$ such that
  \begin{equation}
    \left( \sum_{v_k \in W|_U}c_{k}
      v_{k}\right)_j \le 0,  
    \label{app:stiemke}
  \end{equation}
  for each $j \in \{1,\dots,m\}$, and such that the inequality is strict for
  at least one $j$.

  For $v_k \in W|_U$, let $m_{k}$ denote the number of vectors of $W$ that reduce to it (recall that the operation $\cdot|_U$ need not be injective).  Define the function $M:\R^N_{\ge 0} \to \R$ by
  \begin{equation*}
    M(x) \eqdef  \prod_{u_k \in W}
      \left(x^{u_k} \right)^{c_{k}/m_{k}},
  \end{equation*}
  where $c_{k}$ and $m_{k}$ are chosen for $u_k \in W$ if $u_k|_U = v_k\in W|_U$.  
  Note that, by construction and by the definition of partitioning along a sequence, if $u_k \in W$, then there are $y_j,y_{\ell} \in T_i$ for some $i$, such that $u_k = y_{\ell} - y_{j}$ and
  \begin{equation*}
    \frac{1}{C} \le x_{n}^{u_k} =  \frac{x_{n}^{y_{\ell}}}{x_{n}^{y_j}} \le C,
  \end{equation*}
  for all $n\ge 1$.  Therefore, the sequence $M(x_{n})$ is uniformly, in $n$, bounded both from 
  above and below.    Noting that each $x_n$ has strictly positive components, we may take logarithms and find
  \begin{align*}
    \ln(M(x_{n})) = \bigg( \sum_{u_k\in W}
    \frac{c_{k}}{m_{k}} u_k \bigg) \cdot \ln x_{n},
  \end{align*}
  where for a vector $u \in \R^N_{>0}$ we define
  \begin{align*}
  	\ln(u) \eqdef (\ln(u_1),\cdots, \ln(u_N)).
  \end{align*}
  Expanding along elements of $U$ and $U^c$ yields,
  \begin{align}
  \begin{split}
     \ln(M(x_{n})) &= \bigg( \sum_{v_k\in W|_U} c_{k} v_k \bigg) \cdot \ln (x_{n}|_U) +  \bigg(  \sum_{u_k\in W} \frac{c_{k}}{m_{k}} u_k|_{U^c} \bigg) \cdot \ln (x_{n}|_{U^c}).
     \end{split}
     \label{eq:bound}
  \end{align}
  By construction, $x_{n,\ell}$ is bounded from both above and below
  for $\ell \in U^c$. Thus, the second term in \eqref{eq:bound} is bounded from above and below. By the inequality \eqref{app:stiemke}, where at least one term is strictly negative, and the fact that $x_{n,j} \to 0$ for \textit{each} $j \in U$ along this sequence, we may conclude that the first term, and hence $\ln(M(x_{n}))$ itself, is unbounded towards $+\infty$, as $n \to \infty$.  This is a
  contradiction with the previously found fact that the sequence $M(x_{n})$ is
  uniformly bounded above and below, and the result is shown.
\end{proof}

\subsection{Persistence and the Global Attractor Conjecture in the single linkage class case}
\label{sec:results2}
%
%We now return our attention to a weakly reversible chemical reaction network with mass-action kinetics, $\{\S,\C, \Re, \K\}$.  Letting $\phi(t,x_0)$ denote a trajectory with initial condition $x_0\in \R^{|S|}_{>0}$ and $\omega(x_0)$ the set of $\omega$-limit points, we will assume that the system satisfies the following two conditions
%\begin{enumerate}
%  \item $\phi(t,x_0)$ is bounded in $t$ for any $x_0 \in \R^{|S|}_{>0}$, and
%  \item $\omega(x_0)$ is either completely contained in $\partial
%    \R^{|\S|}_{\ge 0}$ or completely contained within the interior of $ \R^{|\S|}_{>
%      0}$.
%  \end{enumerate}
%We note that if $\{\S,\C, \Re\}$ is weakly reversible and has a deficiency of zero, or, more generally, if $\{\S,\C, \Re, \K\}$ is complex-balanced, then the two conditions are automatically satisfied \cite{FeinbergLec79, Horn72,
%  HornJack72}.  Our goal of this section is to show that if the reaction diagram of $\{\S,\C,\Re\}$ consists of a single linkage class, then the system is persistent in the sense of Definition \ref{def:persistence}.  Thus, in the complex-balanced case, we would have that the conclusion of the Global Attractor Conjecture holds.  
%  
%  

For any $\overline x \in \R^N_{>0}$, define $V_{\overline
  x}:\R^N_{>0} \to \R_{\ge 0}$  by
\begin{equation}
  V_{\overline x}(x) \eqdef \sum_{i = 1}^N \left[ x_i (\ln(x_i) -\ln(\overline
  x_i) - 1) + \overline x_i \right].
  \label{eq:Lyapunov} 
\end{equation}
For $x\in \partial \R^N_{\ge 0}$, define $V_{\overline x}(x)$ via the continuous extension of \eqref{eq:Lyapunov}.
This is the standard Lyapunov function of chemical
reaction network theory \cite{FeinbergLec79, Gun2003} and has commonly been used in the study of  complex-balanced systems where $\overline x$ is typically taken to be a complex-balanced equilibrium.  However, we emphasize that in the current setting $\overline x$ need \textit{not} be a complex balanced equilibrium.   Note that $\nabla V_{\overline x}(x) = \ln x - \ln \overline x$. 
It is relatively straightforward to show that for any $\overline x \in \R^N_{>0}$, $V_{\overline x}$ is convex with a global minimum of zero at $\overline x$ \cite{FeinbergLec79}.

\vspace{.125in}

% It gives a surprising result in that it says that in many instances, whenever condition $C2$ does not hold, see Lemma \cite{lem:main}, the above defined $V_{\overline x}$ acts as a Lyapunov function for \underline{any} $\overline x \in \R^N_{>0}$, whereas in the classical theory it only acts as a Lyapunov function for special $\overline x$ (the complex-balanced equilibria).  This is especially surprising considering that in the following theorem there is no major assumption other than weak-reversibility.

The outline of the technical arguments to come is as follows.  In Lemma \ref{lem:main} we will show that for a trajectory of a non-autonomous, weakly reversible system with bounded kinetics to approach the boundary, at least one of two conditions, $C1$ or $C2$, must hold.  In Lemma \ref{lem:no_union} we will show that condition $C2$ can not hold for trajectories of systems of interest in this paper.  Condition $C1$ of Lemma \ref{lem:main} will then essentially tell us that a whole {\em family} of Lyapunov functions decrease along the trajectory.  This is a substantially stronger condition than having a single Lyapunov function decrease along a trajectory, and will eventually allow us, with the help of Lemma \ref{lem:single}, to overcome the fact that $V_{\overline x}(x)$ does not diverge to $+\infty$ as $x \to \partial \R^N_{>0}$, which has been the technical sticking point to a proof of the Global Attractor Conjecture in the past.  

\begin{lemma}
  Let $\{\S,\C,\Re,\K(t)\}$, with $\S = \{S_1,\dots, S_N\}$, be a weakly reversible,
  non-autonomous mass-action system with bounded kinetics.    For $t\ge 0$, let $x(t) = \phi(t,x_0)$ be the solution to the system with initial condition $x_0$.
  Suppose $x_0 \in \R^N_{>0}$ is such that $\phi(t,x_0)$ remains bounded and $\text{\em dist}(\phi(t,x_0),\partial \R^N_{\ge 0}) \to 0$, as $t \to \infty$.
  %, and $\omega(\phi(t,x_0))\cap \partial \R^N_{> 0} \ne \emptyset$.  
  Then at least one of the following two conditions hold for this trajectory:
  
  \begin{enumerate}
  
  \item[C1:] For any $\overline x \in \R^N_{>0}$, there exists a
    $T = T_{\overline x} >0$ such that $t>T$ implies
    \begin{equation*}
      \frac{d}{dt}V_{\overline x}(x(t)) =  
      \sum_k \kappa_k(t) x(t)^{y_k}(y_k' - y_k)\cdot (\ln(x(t)) - \ln(\overline x))< 0.
    \end{equation*}

  \item[C2:] There exists a sequence of times, $t_n \to \infty$, such that $x_n \eqdef \phi(t_n,x_0) \in
    \R^{N}_{>0}$ converges to a point $z\in \omega(\phi(\cdot ,x_0))
    \cap \partial \R^N_{\ge 0},$ and
    \begin{enumerate}[$(i)$]
    \item $\C$ is partitioned along $x_n$ with tiers
      $\{T_i\}_{i=1}^P$, and constant $C$, and
    \item $T_1$ consists of a union of linkage classes.
    \end{enumerate}
  \end{enumerate}
  \label{lem:main}
\end{lemma}

\begin{proof}
  We suppose condition $C1$ does not hold, and will conclude that condition $C2$ must then hold.  
  Because condition $C1$ does not hold, there is an $\overline x\in \R^N_{>0}$ and a sequence $t_n \to \infty$ such that
  \begin{align}
    \sum_k \kappa_k(t_n) x_n^{y_k}(y_k' - y_k) \cdot (\ln (x_n) -
    \ln(\overline x)) \ge 0,
    \label{eq:bound1}
  \end{align}
  where $x_n = \phi(t_n,x_0)$. We now fix $\overline x$.   
  
  Combining   $\text{dist}(\phi(t_n,x_0),\partial \R^N_{\ge 0}) \to 0$, as $t_n \to \infty$, with  the boundedness of the trajectory allows us to conclude that there exists a
  convergent subsequence of $\{x_n\}$, which we take to be the
  sequence itself, with limit point $z \in \omega(\phi(\cdot,x_0))
  \cap \partial\R^N_{>0}$.  Note that by construction the inequality \eqref{eq:bound1} holds for
  all $x_n$ of the subsequence.    Applying Lemma \ref{lem:partition}, we partition the complexes along an
  appropriate subsequence of the sequence, which we will again denote $\{x_n\}$, with tiers $T_i$, $i = 1,\dots,P$, and constant $C>1$.  
  
  In the following, for tier $i\in \{1,\dots,P\}$, we denote by 
  \begin{itemize}
  \item $\{i \to i\}$ all reactions with both source and product complex in $T_i$,
  \item $\{i \to i + m\}$ all reactions with source complex in $T_i$ and product complex in $T_{i+m}$ for $m \in \{1,\dots,P-i\}$,
  \item $\{i \to i - m\}$ all reactions with source complex in $T_i$ and product complex in $T_{i-m}$ for $m \in \{1,\dots , i-1\}$.
  \end{itemize}
  Defining $u/v \eqdef (u_1/v_1,\dots, u_N/v_N)$ for $u,v \in \R^N_{>0}$, we may re-write the left hand side of the inequality  \eqref{eq:bound1}
  \begin{align}
    \sum_k \kappa_k(t_n) x_{n}^{y_k}(y_k' - y_k)\cdot \ln \left(\frac{x_n}{\overline x} \right) 
    %&= \sum_{i = 1}^P\bigg[ \sum_{\{i
%      \to i\}} \kappa_k(t_n) x_{n}^{y_k}(y_k' - y_k)\cdot \ln\left(
%      \frac{x_n}{\overline x}\right) \notag \\
%    &\hspace{.2in} + \sum_{m=1}^{P-i} \sum_{\{i \to i + m\}} \kappa_k(t_n)
%    x_{n}^{y_k}(y_k' - y_k)\cdot \ln\left( \frac{x_n}{\overline
%        x}\right) \notag \\
%    &\hspace{.2in} + \sum_{m=1}^{i-1} \sum_{\{i \to i - m\}} \kappa_k(t_n)
%    x_{n}^{y_k}(y_k' - y_k)\cdot \ln\left( \frac{x_n}{\overline
%        x}\right) \bigg]\notag \\
    &= \sum_{i = 1}^P\bigg[ \sum_{\{ i \to i \}} \kappa_k(t_n)
    x_{n}^{y_k}\bigg[ \ln \left(\frac{x_n^{y_k'}} {x_n^{y_k}}\right)
      +c_k\bigg]\label{eq:1} \\
    &\hspace{.2in}+ \sum_{m=1}^{P-i} \sum_{\{ i \to i + m\}} \kappa_k(t_n)
    x_{n}^{y_k}\bigg[ \ln \left(\frac{x_n^{y_k'}} {x_n^{y_k}}\right)
      +c_k\bigg]\label{eq:2} \\
    &\hspace{.2in} + \sum_{m=1}^{i-1} \sum_{\{i \to i - m\}} \kappa_k(t_n)
    x_{n}^{y_k}\bigg[ \ln \left(\frac{x_n^{y_k'}} {x_n^{y_k}}\right)
      +c_k\bigg]\bigg],\label{eq:3}
  \end{align}
  where for the $k$th reaction $c_k \eqdef  \ln({\overline x}^{y_k}/{\overline
          x}^{y_k'}) = -(y_k' - y_k)\cdot \ln \overline x$. 
  Note that $\sup_k|c_k| < \infty$ because $\overline x$ is fixed.  Note also that, by construction, for large enough $n$ any component in the enumeration \eqref{eq:2} is negative, and, in fact, $\ln(x_n^{y_k'}/x_n^{y_k}) \to -\infty$ as $n \to \infty$, for these terms.   We will now show that the total summation above (that is, the left hand side of \eqref{eq:1} and, hence, \eqref{eq:bound1}) must also, for large enough $n$, be strictly negative unless condition $C2$ holds.  This will then conclude the proof as it shows ``not $C1$ $\implies$ $C2$.''

  Suppose condition $C2$ does not hold.  Then, for the specific partition we have along $\{x_n\}$, it must be that $T_1$ \textit{does not} consist of a union of linkage classes.  Thus, by the weak reversibility of the system 
  %(and the fact that no linkage class can consist of a single complex), 
  there must be at least one reaction, $y_{k1} \to y_{k1}'$ say, such that $y_{k1} \in T_1$ and $y_{k1}' \in T_j$ for $j \ge 2$.  That is, there is a reaction  being enumerated in \eqref{eq:2} with $i = 1$ and $m \ge 1$.  As noted above, we have by construction that $\ln(  x_n^{y_{k1}'}/x_n^{y_{k1}}) \to - \infty,$  as $n \to \infty$.  Further, there is a $C>1$ such that for any $y_j \in \C$, $x_n^{y_{k1}} \ge (1/C) x_n^{y_j}$ for all $n$.  Finally, the terms $\ln(x_n^{y_k'}/x_n^{y_k})$ on the right hand side of  \eqref{eq:1} are uniformly bounded in $n$ and $k$.  Combining these ideas shows that for $n$ large enough
  \begin{align}
     \frac{1}{2}\kappa_{k1}(t_n)    x_{n}^{y_{k1}}\bigg[ \ln \left(\frac{x_n^{y_{k1}'}} {x_n^{y_{k1}}}\right)
      +c_{k1}\bigg] + \sum_{i = 1}^P \sum_{\{ i \to i \}} \kappa_k(t_n)
    x_{n}^{y_k}\bigg[ \ln \left(\frac{x_n^{y_k'}} {x_n^{y_k}}\right)
      +c_k\bigg] \le - q_{1,n} x_n^{y_{k1}},
      \label{eq:dom}
  \end{align}
  where $q_{1,n} \to \infty$ as $n \to \infty$.  Thus, we have used one half of a single term enumerated in \eqref{eq:2} to bound every term enumerated on the right hand side of \eqref{eq:1}.

  Since we already know that the terms in \eqref{eq:2} are all strictly negative for large enough $n$, all that remains to show is that the terms in \eqref{eq:3}, which are all positive for large enough $n$, are also dominated, in a similar manner as \eqref{eq:dom}, by some terms in \eqref{eq:2}, as $n \to \infty$.
  
  Pick a reaction from  \eqref{eq:3}, $y_0 \to y_0'$, say.  Suppose that the source of the reaction is  a complex in tier $T_i$, and the product is in tier $T_{i-m}$ for some $m>0$.  By the weak reversibility of the network, there is a series  of reactions beginning with $y_0'$ and ending with $y_0$ such that no reaction is enumerated more than once (that is, there is a path along the directed diagram with no $y_k \to y_k'$ used multiple times).  We now claim that there must be a subset of these reactions, enumerated as $r_1,\dots, r_b$, satisfying the following conditions (below, we denote the tier of the source complex for reaction $r_j$ as $T_{d_{j,s}}$ and the tier for the corresponding product complex as $T_{d_{j,p}}$):
  \begin{enumerate}
  	\item %The source complex of reaction $r_1$ is in tier $T_{d_{1,s}}$ with 
	$d_{1,s} \le i-m$.
	\item For $\ell \in \{1,\dots, b\}$, the source complex of $r_{\ell}$ is in a strictly higher tier than the corresponding product complex; that is, $d_{\ell,s} < d_{\ell,p}$.
	\item For $\ell \ge 2$, we have $d_{\ell,s} \le d_{\ell-1,p}$.
	\item %The product complex of reaction $r_b$ is in a tier that is lower than or equal to tier $i$; that is, 
	$d_{b,p} \ge i$.  
	\item For $\ell \in \{1,\dots, b\}$, we have $d_{\ell,s} < i$.
  \end{enumerate}
  Note that, for example, the series of reactions $r_1,\dots, r_b$ above can be constructed from the original series by only taking the first $b$ reactions for which the source is in a strictly higher tier than the product, but stop (i.e. pick $b$) once a product complex is in a tier that is equal to or lower than that of $y_0$.  If  $y_0 \to y_0'$ is a reversible reaction, then we may take $b=1$ with the reaction $r_1$ simply being the reverse reaction $y_0' \to y_0$ from tier $T_{i-m}$ to tier $T_i$. 
  
  By condition 2 we know that each of the reactions $r_1,\dots,r_b$ are enumerated in \eqref{eq:2}. %By condition 5 we know that $x_n^{y_{r_{\ell}}}/x_n^{y_0} \to \infty, \quad \text{as } n\to \infty,$  for each $\ell \in \{1,\dots,b\}$.  
  Let $d_0 = \max_{\ell}\{d_{\ell,s}\}$, and let $y_{d_0} \in T_{d_0}$ be a choice of complex from tier $T_{d_0}$. Note that $d_0 < i$ by condition 5, and so 
  \begin{equation}
  x_n^{y_{d_0}}/x_n^{y_0} \to \infty, \quad \text{as } n \to \infty.
  \label{eq:bound5}
 \end{equation}
   By construction and an application of the triangle inequality, there is a constant $C_1>0$ such that for $n$ large enough (so that the $\ln$ terms dominate the $c_{r_{\ell}}$ terms)
  \begin{align}
  	\bigg| \sum_{\ell = 1}^b  \kappa_{r_{\ell}}(t_n) x_n^{y_{r_{\ell}}}  \bigg[ \ln \bigg(\frac{x_n^{y_{r_{\ell}}'}} {x_n^{y_{r_{\ell}}}}\bigg)
      +c_{r_{\ell}}\bigg] \bigg| 
      &\ge  \eta C_1 x_n^{y_{d_0}} \bigg[\ln \bigg( \prod_{\ell = 1}^b \frac{x_n^{y_{r_{\ell}}}} {x_n^{y'_{r_{\ell}}}} \bigg) - \bigg| \sum_{\ell = 1}^b c_{r_{\ell}}\bigg| \bigg],\label{eq:product}
  \end{align}
  where the apparent ``flip'' of the ratio comes from taking the absolute value of a negative term, and where $\eta>0$ is the parameter used to bound all the functions $\kappa_k(t)$:  $\eta < \kappa_k(t) < 1/\eta$, for all $t \ge0$.
  By condition 2 in the construction above we have that each term in the product on the right hand side of \eqref{eq:product} goes to $+\infty$, as $n\to \infty$, and hence the entire product does.  Further, by conditions 1, 3, and 4 above, there is a $C_2>0$ such that 
  \begin{equation*}
  	\prod_{\ell = 1}^b \frac{x_n^{y_{r_{\ell}}}} {x_n^{y'_{r_{\ell}}}} \ge C_2 \frac{x_n^{y_0'}}{x_n^{y_0}},
  \end{equation*}
  uniformly in $n$. 
  % In fact, if in the construction above, we have that $d_{b+1} > i$, then the product asymptotically dominates $x_n^{y_0'}/x_n^{y_0}$.  
  We now may conclude that there is a $C_3>0$ so that
  \begin{align}
  \bigg[\ln \bigg( \prod_{\ell = 1}^b \frac{x_n^{y_{r_{\ell}}}} {x_n^{y'_{r_{\ell}}}} \bigg) - \bigg| \sum_{\ell = 1}^b c_{r_{\ell}}\bigg| \bigg] \ge C_3 \ln ( x_n^{y_0'}/ x_n^{y_0})  +c_0,
  \label{bnd}
  \end{align}
  for all $n$.  Let $M$ be the number of reactions enumerated in \eqref{eq:3} over all $i \in\{1,\dots,P\}$. Combining \eqref{bnd} with \eqref{eq:product} and \eqref{eq:bound5}, we see that for $n$ large enough
  \begin{equation}
   \frac{1}{2M} \sum_{\ell = 1}^b  \kappa_{r_{\ell}}(t_n) x_n^{y_{r_{\ell}}}  \bigg[ \ln \bigg(\frac{x_n^{y_{r_{\ell}}'}} {x_n^{y_{r_{\ell}}}}\bigg)
      +c_{r_{\ell}}\bigg] + \kappa_0(t_n)
    x_{n}^{y_0}\bigg[ \ln \left(\frac{x_n^{y_0'}} {x_n^{y_0}}\right)
      +c_0\bigg] \le - q_{2,n} x_n^{y_{d_0}},
      \label{eq:final}
 \end{equation}
 where $q_{2,n} \to \infty$,  as $n \to \infty$. Note that  the first term on the left hand side of \eqref{eq:final} is the left hand side of \eqref{eq:product} divided by $2M$, and the second term is the summand of \eqref{eq:3} associated with the reaction $y_0 \to y_0'$.  As $y_0 \to y_0'$ was arbitrary, we may conclude that for $n$ large enough,
 \begin{align}
 \begin{split}
   &\frac{1}{2} \sum_{i = 1}^P \bigg[ \sum_{m=1}^{P-i} \sum_{\{ i \to i + m\}} \kappa_k(t_n)
    x_{n}^{y_k}\bigg[ \ln \left(\frac{x_n^{y_k'}} {x_n^{y_k}}\right)
      +c_k\bigg] \\
      &\hspace{.2in} +\sum_{m=1}^{i-1} \sum_{\{i \to i - m\}} \kappa_k(t_n)
    x_{n}^{y_k}\bigg[ \ln \left(\frac{x_n^{y_k'}} {x_n^{y_k}}\right)
      +c_k\bigg]\bigg]<0
      \end{split}
      \label{eq:final2}
 \end{align}
which are the terms of \eqref{eq:2} plus the terms of \eqref{eq:3}.

  Combining the inequalities \eqref{eq:dom} and \eqref{eq:final2} shows that for $n$ large enough, the summation found on the left hand side of \eqref{eq:1}, or equivalently on the left hand side of \eqref{eq:bound1}, must be strictly negative. This is a contradiction with \eqref{eq:bound1} holding for all $n$.  Therefore, we must have that condition $C2$ holds.  
\end{proof}

  \begin{lemma}
  	 Let $\{\S,\C,\Re,\K(t)\}$, with $\S = \{S_1,\dots, S_N\}$, be a %weakly reversible, single linkage class,
  non-autonomous mass-action system with bounded kinetics and a single linkage class.  
  Suppose $x_0 \in \R^N_{>0}$ is such that $\phi(t,x_0)$ remains bounded and $\text{\em dist}(\phi(t,x_0),\partial \R^N_{\ge 0}) \to 0$ as $t \to \infty$.     Then, there does not exist a subsequence of times $t_n\to \infty$ such that $\C$ is partitioned along $x_n \eqdef \phi(t_n,x_0)$ in which $T_1$ consists of a union of linkage classes.
  \label{lem:no_union}
  \end{lemma}
  
  \begin{proof}
  	Note that in the one linkage class case $T_1$ can only consist of a union of linkage classes if $T_1 \equiv \C$. We suppose there is such a sequence of times, $t_n\to \infty$, such that $\C$ is partitioned along $x_n \eqdef \phi(t_n,x_0)$ with $T_1 \equiv \C$ (and there are no other tiers). By the boundedness of the sequence, we may consider a convergent subsequence with limit point $z$.  Let $U = U(z) = \{i \in \{1,\dots, N\} \ : \ z_i = 0\}.$  Note that $U$ is nonempty because the trajectory goes to the boundary.  Note that $\C$ is necessarily partitioned along this subsequence as well, with the same tiers.  By Theorem \ref{thm:conservation} there is a
  non-negative conservation relation $w \in \R^N_{\ge 0}$ that
  respects the pair $(U,\{T_i\})$.  
  
   Because the support of $w$ is $U \ne \emptyset$, and $\phi_i(t_n,x_0) \to 0$ for all $i \in U$, we have that $w\cdot \phi(t_n,x_0) \to 0$, as $n\to \infty$.  However, because $T_1 \equiv \C$, we also have that $w \cdot (y_k' - y_k) = 0$ \textit{for all} $y_k \to y_k' \in \Re$.  Thus, we see from \eqref{eq:main_general} that $w \cdot \phi(t,x_0) >0$ is constant in time $t$, contradicting that $w\cdot \phi(t_n,x_0) \to 0$ as $n \to \infty$.
  \end{proof}

Lemma \ref{lem:no_union} says that condition $C2$ of Lemma \ref{lem:main} can never hold in our setting, and so $C1$ must.  The following, final lemma will allow us to conclude that any trajectory satisfying such a family of Lyapunov functions must converge to a single point.

\begin{lemma}
  Let $\{\S,\C,\Re,\K(t)\}$, with $\S = \{S_1,\dots, S_N\}$, be a %weakly reversible, 
  non-autonomous system
  with bounded mass-action kinetics. 
  %Suppose $x_0 \in \R^N_{>0}$ is such that $\omega(\phi(t,x_0))\cap \partial \R^N_{> 0} \ne \emptyset$.   
  Suppose $x_0 \in \R^N_{>0}$ is such  that for any $\overline x \in \R^N_{>0}$, there exists a $T =
  T_{\overline x} >0$ such that $t>T$ implies
  \begin{equation*}
    \frac{d}{dt}V_{\overline x}(x(t)) < 0,
  \end{equation*}
  where  $x(t) = \phi(t,x_0)$ is the solution to the system with $x(0) = x_0$ and kinetics $\K(t)$.  Then $\omega(\phi(\cdot,x_0))$ is a single point.
  \label{lem:single}
\end{lemma}
  
\begin{proof}
  Note that the trajectory remains bounded because each $V_{\overline x}(x(t))$ does.  Also, we have that for any $\overline x \in \R^N_{>0}$ there exists a $c_{\overline x} \ge 0$ such that
  \begin{equation*}
  	V_{\overline x}(x(t)) \to c_{\overline x}, \quad \text{as } t \to \infty,
  \end{equation*}
  where the non-negativeness of $c_{\overline x}$ follows by the fact that $V_{\overline x}(x) \ge 0$ for all $x \in \R^N_{\ge 0}$.
  The boundedness of $x(t)$ implies there is at least one $\omega$-limit point of the trajectory. The question now is: can there be more than one?  Suppose so.   That is, we assume the existence of $z_1,z_2 \in  \omega(\phi(\cdot,x_0))$ with $z_1 \ne z_2$.
  
%  So long as $\omega(\phi(t,x_0))$ is not identically equal to the origin (which we know it is not as we are assuming it contains more than one point), we may select  $z_1, z_2 \in \omega(\phi(t,x_0))$ with $z_1 \ne z_2$ and $z_{1,i} = 0 \iff z_{2,i} = 0$.  This follows from the fact that $\omega(\phi(t,x_0))$ is connected.  After reordering the indices, let $\{1,\dots,n\}$ denote the indices for which $z_{1,i} \ne z_{2,i}$.  
  
  Note that for any $\overline x\in \R^N_{>0}$, $V_{\overline x}(z_1) = V_{\overline x}(z_2) = c_{\overline x}$, where if $z_i \in \partial \R^N_{\ge 0}$ we define $V_{\overline x}$ on the boundary via its continuous extension to the boundary.  Let $\overline x_1 , \overline x_2 \in \R^N_{>0}$ be arbitrary.  Then, after some algebra we have
  \begin{align*}
  	0 &= V_{\overline x_1}(z_1) - V_{\overline x_1}(z_2) - (V_{\overline x_2}(z_1) - V_{\overline x_2}(z_2))\\
%	&= \sum_{i = 1}^N z_{1,i} (\ln(z_{1,i}) -\ln(\overline x_{1,i}) - 1) -  \sum_{i = 1}^N z_{2,i} (\ln(z_{2,i}) -\ln(\overline x_{1,i}) - 1)\\
%	&\hspace{.3in} - \left(\sum_{i = 1}^N z_{1,i} (\ln(z_{1,i}) -\ln(\overline x_{2,i}) - 1) -  \sum_{i = 1}^N z_{2,i} (\ln(z_{2,i}) -\ln(\overline x_{2,i}) - 1) \right)\\
	%&= \sum_{i = 1}^n z_{1,i}(\ln(\overline x_{2,i}) - \ln(\overline x_{1,i})) - \sum_{i = 1}^n z_{2,i}(\ln(\overline x_{2,i}) - \ln(\overline x_{1,i}))\\
	&= (z_1 - z_2)\cdot \left( \ln(\overline x_{2}) - \ln(\overline x_{1}) \right).
  \end{align*}
  But, $\overline x_1, \overline x_2\in \R^N_{>0}$, and hence $(\ln \overline x_2 - \ln \overline x_1)\in \R^N$, were
  arbitrary.  Thus, $z_1 = z_2$.
\end{proof}

  We now have our main result.
  
\begin{theorem}
  Let $\{\S,\C,\Re,\K\}$, with $\S = \{S_1,\dots, S_{|\S|}\}$, be a weakly reversible, single linkage class chemical
  reaction network with mass-action kinetics.  
   We assume that for $x_0 \in \R^{|\S|}_{>0}$ the trajectory $\phi(t,x_0)$ satisfies the following two conditions
\begin{enumerate}
  \item $\phi(t,x_0)$ is bounded (in $t$), and
  \item $\omega(\phi(\cdot,x_0))$ is either completely contained in $\partial
    \R^{|\S|}_{\ge 0}$ or completely contained within the interior of $ \R^{|\S|}_{>
      0}$.
  \end{enumerate}
  Then   $\omega(\phi(\cdot,x_0)) \cap \partial \R^{|\S|}_{\ge 0} = \emptyset$, and the trajectory is persistent. 
  \label{thm:main}
\end{theorem}

\begin{remark}
	Note that the conclusion of the theorem guarantees that $\omega(\phi(\cdot,x_0))$ is completely contained within the interior of the strictly positive orthant.  Said differently, $\omega(\phi(\cdot,x_0))$ \textit{can not} be contained within $\partial \R^{|\S|}_{\ge 0}$.
\end{remark}

\begin{proof}
  Suppose, in order to find a contradiction, that for this $x_0$ there is at least one $z \in \omega(\phi(\cdot,x_0))\cap \partial \R^{|\S|}_{\ge 0}$.   Let 
\begin{equation*}
   U =  \{i \in \{1,\dots,|\S|\} \, : \,  z_i = 0 \text{ for some } z \in \omega(\phi(\cdot,x_0))\},
   \end{equation*}
   which is nonempty as there is at least one $z \in \partial \R^{|\S|}_{\ge 0}$.
 That is, these are all the indices for the species whose concentrations approach zero along some subsequence of times for this specific, fixed trajectory.  Therefore, and equivalently, $i \in U$ if and only if
 \begin{equation*}
     \liminf_{t\to \infty} \phi_i(t,x_0) = 0 \quad \text{ and } \quad \limsup_{t\to \infty} \phi_i(t,x_0) < \infty,
  \end{equation*}
  where the second fact follows from the boundedness of trajectories, whereas for $j \notin U$ we have 
  \begin{equation}
  0< \liminf_{t\to \infty} \phi_j(t,x_0)\le  \limsup_{t\to \infty} \phi_j(t,x_0) < \infty.
  \label{eq:other_bounds}
  \end{equation}
  
   Let $\{\S_U,\C_U,\Re_U\}$ denote the reduced reaction network of $\{\S,\C,\Re\}$ associated with $U$ (see Definition \ref{def:reduced}), and let $\K(t) = \K_U(t)$ denote the projected dynamics (see Section \ref{sec:reduced_networks}), with $\kappa_k(t)$ denoting the non-autonomous variables defined via \eqref{eq:new_rates} that take the place of the rate constants in standard mass-action kinetics.  It is important to note that by \eqref{eq:other_bounds} and the definition of the $\kappa_k(t)$'s given in \eqref{eq:new_rates}, we have the existence of an $\eta>0$ such that
 \begin{equation}
 	\eta < \kappa_k(t) < 1/\eta,
	\label{eq:important_bound}
 \end{equation}
 for all $t \ge 0$ and all  $k \in \{1,\dots, |\Re|_U|\}$.  That is, $\{\S_U,\C_U,\Re_U,\K(t)\}$ is a non-autonomous system with bounded mass-action kinetics.  Another way to see this fact is just to note that the chemical species whose concentrations are uniformly bounded from above and below, those $j \notin U$, have been incorporated into the $\kappa_k(t)$, thereby yielding a bound like \eqref{eq:important_bound}.  By Lemmas \ref{lem:linkage} and \ref{lemma:WR}, the reduced network $\{\S_U,\C_U,\Re_U\}$ is weakly reversible and has only one linkage class.
 
 We let $|\S_U| = N$ and denote by $x(t)\in \R^N_{>0}$ the solution to the reduced dynamical system for this specific trajectory.  By condition 2., above, which pertains to the original system, and by the construction of $U$, the set of $\omega$-limit points of the trajectory of the reduced system must exist on $\partial \R^N_{\ge 0}$.  Combining Lemmas \ref{lem:linkage}, \ref{lem:main}, \ref{lem:no_union}, and \ref{lem:single} shows that the set of $\omega$-limit points of the trajectory of the reduced system, $x(t)$, must consist of a single point.  By construction of $U$, this point must be the origin $\vec 0 \in \R^N$, as otherwise there is an $i \in  U$ for which $\liminf_{t \to \infty} x_i(t) > 0$, a contradiction with the definition of $U$.  However, we also know by Lemmas \ref{lem:main}, \ref{lem:no_union}, and \ref{lem:single} that $\frac{d}{dt} V_{\overline x}(x(t)) < 0$ for $t$ large enough, where $\overline x\in \R^N_{>0}$ is arbitrary.  Therefore, because the origin is a local maximum of $V_{\overline x}$, we can not have that $x(t) \to \vec 0\in \R^N$.
\end{proof}

The following corollary, which was the main goal of the paper, states that the Global Attractor Conjecture holds in the single linkage class case. 

\begin{corollary}
	Let $\{\S,\C,\Re,\K \}$ denote a complex-balanced system with one linkage class.   Then, any complex-balanced equilibrium contained in the interior of a positive compatibility
class  is a global attractor of the interior of that positive class.
	\label{cor:GAC}
\end{corollary}
\begin{proof}
	Trajectories of complex-balanced systems satisfy conditions 1 and 2 in the statement of Theorem \ref{thm:main}, \cite{FeinbergLec79}.  The result then follows by the discussion in Section \ref{sec:background} after the statement of the Global Attractor Conjecture.
\end{proof}

In particular, if $\{\S,\C,\Re,\K \}$ is a weakly reversible, deficiency zero system with a  single linkage class, then the conclusion of the Global Attractor Conjecture holds.

	Note that the single linkage class assumption in Theorem \ref{thm:main}, and hence Corollary \ref{cor:GAC}, was only used in conjunction with Theorem \ref{thm:conservation} in the proof of Lemma \ref{lem:no_union} to guarantee that the top tier, $T_1$, could not consist of a union of linkage classes.  If it can be guaranteed in any other way that the top tier, in the construction detailed in the previous lemmas, can not consist of a union of linkage classes, then the conclusions of Theorem \ref{thm:main} and Corollary \ref{cor:GAC}, that complex-balanced equilibria are global attractors of their positive classes, will still hold.  Also, note that if it can be shown that condition 2 of Theorem \ref{thm:main} is always satisfied by weakly reversible networks with mass-action kinetics, something we believe to be true, then the Persistence Conjecture, as stated in Section \ref{sec:background} of this paper, will also be proven in the single linkage class case by the arguments in this paper.

	For completeness, we state the following corollary to the proof of Theorem \ref{thm:main}, which points out that the rate constants $\kappa_k$ are permitted to be bounded functions of time before the projection.
	\begin{corollary}
  Let $\{\S,\C,\Re,\K(t)\}$, with $\S = \{S_1,\dots, S_{|\S|}\}$, be a weakly reversible, single linkage class chemical
  reaction network with bounded mass-action kinetics.  
   We assume that for $x_0 \in \R^{|\S|}_{>0}$ the trajectory $\phi(t,x_0)$ satisfies the following two conditions:
\begin{enumerate}
  \item $\phi(t,x_0)$ is bounded (in $t$), and
  \item $\omega(\phi(\cdot,x_0))$ is either completely contained in $\partial
    \R^{|\S|}_{\ge 0}$ or completely contained within the interior of $ \R^{|\S|}_{>
      0}$.
  \end{enumerate}
  Then   $\omega(\phi(\cdot,x_0)) \cap \partial \R^{|\S|}_{\ge 0} = \emptyset$, and the trajectory is persistent. 
  \label{cor:main}
\end{corollary}
\begin{proof}
The proof is exactly the same as the proof of Theorem \ref{thm:main}.
\end{proof}

\section{Example}
\label{sec:example}

We present one example to demonstrate the ease with which the main results can be applied.  Consider the system with reaction network 
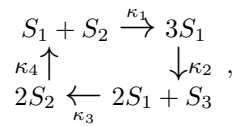
\begin{equation*}
\begin{array}{c}
	S_1 + S_2\ \overset{\kappa_1}{\text{{\Large $\rightarrow$}}} \ 3S_1\\
	\text{\small $\kappa_4$} \text{{\Large $\uparrow$}} \hfill \text{{\Large $\downarrow$}} \text{\small $\kappa_2$}\\
	2S_2 \ \underset{\kappa_3}{\text{\Large $\leftarrow$}} \ 2S_1 + S_3
\end{array},
\end{equation*}
which was recently used as an example of a system whose persistence was beyond the scope of known theory \cite{JohnstonSiegel2011}.  This network has four complexes, one linkage class, and the dimension of the stoichiometric subspace can be checked to be three.  Thus, the deficiency is zero.  As the network is clearly weakly reversible, the Deficiency Zero Theorem implies that, regardless of the choice of $\kappa_k$, the system is complex-balanced.  Theorem \ref{thm:main} and Corollary \ref{cor:GAC} tell us that the system is persistent, and  that each of the complex-balanced equilibria are global attractors of their positive stoichiometric compatibility classes.  Thus, the long term behavior of the system is now completely known.

\vspace{.1in}	
	\noindent \textbf{Acknowledgements.}  Thanks to Masanori Koyama for a thorough reading of an early draft that found an error in the proof of Lemma \ref{lem:main}.  Also thanks to two incredibly meticulous reviewers who helped with the presentation substantially.

\bibliographystyle{amsplain} \bibliography{GAC_ONE}
  
\end{document}